\documentclass[preprint,12pt]{elsarticle}

\usepackage{amssymb}
\usepackage{amsmath}
\usepackage{amsfonts}
\usepackage{graphicx,amsmath,amsfonts,amssymb,amscd,amsthm}
\usepackage{epsfig}
\usepackage{epsf}
\allowdisplaybreaks  
\newtheorem{theorem}{Theorem}

\newtheorem{definition}{Definition}
\newtheorem{lemma}{Lemma}

\newtheorem{remark}{Remark}

\newcommand{\x}{\mbox{\boldmath $x$}}

\newcommand{\be}{\begin{equation}}
\newcommand{\ee}{\end{equation}}
\newcommand{\ba}{\begin{align}}
\newcommand{\ea}{\end{align}}

\newcommand{\bea}{\begin{array}}
\newcommand{\eea}{\end{array}}

\def\EquationsBySection{\def\theequation
{\thesection.\arabic{equation}}%
\@addtoreset{equation}{section}}

\newcommand\old[1]{}
\newcommand{\pend}{\hfill \thicklines \framebox(6.6,6.6)[l]{}}
\renewenvironment{proof}{\noindent {\it  Proof.} \rm}{\pend}
\newtheorem{proposition}{Proposition}

\usepackage{indentfirst}
\EquationsBySection \makeatother
\usepackage{indentfirst}
\usepackage{floatrow}
\newfloatcommand{capbtabbox}{table}[][\FBwidth]
\date{2019}

\begin{document}
\begin{frontmatter}



\title{Stochastic viscosity solutions for stochastic integral-partial differential equations and singular stochastic control}


\author{Jinbiao Wu\footnote{ Email: wujinbiao@csu.edu.cn}
 }

\address{ \small\it School of Mathematics and Statistics, \\
 \small\it Central South University,
Changsha 410083, Hunan, China}

\begin{abstract}
In this article, we mainly study stochastic viscosity solutions for a class of semilinear
 stochastic integral-partial differential equations (SIPDEs). We investigate a new class of generalized backward doubly stochastic differential equations (GBDSDEs) driven by two independent Brownian motions and an independent Poisson random measure, which involves an integral with respect to a c\`{a}dl\`{a}g increasing process. We first derive existence and uniqueness of the solution of GBDSDEs with general jumps. We then introduce the definition of stochastic viscosity solutions of SIPDEs and give a probabilistic representation for stochastic viscosity solutions of semilinear SIPDEs with nonlinear Neumann boundary conditions.
Finally, we establish stochastic maximum principles for the optimal control of a stochastic system modelled by a GBDSDE with general jumps.
\end{abstract}


\end{frontmatter}

\section{Introduction}

The theory of viscosity solution for a partial differential equation (PDE) has been first introduced by
 \cite{Crandall1983} in 1983. The theory has brought in-depth impact on the modern theoretical mathematics and also has bloomed into an extremely important tool in various applied mathematical areas, especially in optimal control, finance,  operations research, and so
on. For a review of the main results
and methods, readers may refer  to the well-known papers by Lions \cite{Lions1988}, Crandall et al. \cite{Crandall1992}, Bianchini and Bressan \cite{Bianchini2005} and the books by Fleming and Soner \cite{Fleming2006} and by Bardi and Capuzzo-Dolcetta \cite{Bardi2008}.  Since the seminal works
by Lions and Souganidis \cite{Lions1998, LionsSouganidis1998,Lions2000}, there have been many efforts to develop the theory
of stochastic viscosity solutions. Generally speaking, they
provided two approaches. The first one is to use the so-called stochastic characteristics
(see Kunita \cite{Kunita1997}) to remove the stochastic integrals of stochastic partial differential equations (SPDEs). The stochastic viscosity solution can be studied $\omega$-wisely, either by defining test
functions along the characteristics, or by defining the stochastic viscosity solution via the
transformed $\omega$-wise PDEs (deterministic). Another approach is to approximate the Brownian path $\omega$ by smooth functions and define the weak solution as the limit, if
exist, of the solutions to the approximating equations which are deterministic PDEs.  Buckdahn and Ma \cite{Buckdahn2001, BuckdahnMa2001, Buckdahn2002} were the first to use the first approach to connect the stochastic viscosity solution of SPDEs with backward doubly stochastic differential equations (BDSDEs). The key idea is to  define the stochastic viscosity solution via the Doss-Sussman transformation (see, e.g., \cite{Sussmann1978}) that an SPDE can be converted to an ordinary PDE with random coefficients. Later, Boufoussi et al. \cite{Boufoussi2007} refered to the technique of Buckdahn and Ma \cite{Buckdahn2001}
 to give a probabilistic representation for stochastic viscosity solutions of semi-linear SPDEs with Neumann boundary conditions.
In recent years, Buckdahn et al. \cite{Buckdahn2015} studied viscosity solutions for a fairly large class of fully nonlinear SPDEs which can also be viewed as forward path dependent PDEs or rough PDEs. Ekren et al. \cite{Ekren2016a, Ekren2016b} provided a notion of
viscosity solutions for fully nonlinear parabolic path-dependent PDEs.  Keller and Zhang \cite{Keller2016} gave some results about pathwise It\^{o} calculus for rough paths which are very useful for studying stochastic viscosity solutions of SPDEs driven by rough paths. Qiu \cite{Qiu2018} introduced the notion of viscosity solution to a fully nonlinear stochastic Hamilton-Jacobi-Bellman equation.   Matoussi et al. \cite{Matoussi2019} investigated the links between the second order  BDSDEs and stochastic viscosity solutions of parabolic fully nonlinear SPDEs.
Frankly speaking, it is really tough to deal with stochastic viscosity solutions to SPDEs due to the tedious technical details. However, recent research developments in rough paths theory open another window for studying stochastic viscosity solutions. We hope to be able to address this issume in our future publications.

The purpose of this article is to give a probabilistic representation for the stochastic viscosity solution of the SIPDE with nonlinear Neumann boundary condition:
\begin{align*}
\left\{\begin{array}{l}
du(t,x)+\left[\mathcal{L}u(t,x)+f\left(t,x,u(t,x), (\nabla u\sigma)(t,x), \mathcal{A}u(t,x)\right)\right]dt\\
 \ \ \ \ \ \  \ \ \ \ + g(t,x)\overleftarrow{d}B(t)=0, \ \ \ (t,x)\in [0,T]\times \mho,\\
\frac{\partial u}{\partial \mathbf{n}}(t,x)+h(t,x,u(t,x))=0, \ \ \ (t,x)\in [0,T]\times \partial\mho,\\
u(T,x)=l(x), \ \ \ x\in \bar{\mho},
\end{array}
\right.
\end{align*}
where $B$ is a standard Brownian motion, $\mathbf{n}$ are inner normals at $x$,
$\mho$ is an open connected bounded domain of $\mathbb{R}^n$, $\mathcal{L}$ is an infinitesimal generator of some jump-diffusion process and $\mathcal{A}$ is an integral operator. Here $f,g,h$ and $l$ are some measurable functions.

To give a probabilistic formula for the solution of the above SIPDE with nonlinear Neumann boundary condition, we first study a new class of backward stochastic differential equations called GBDSDEs with general jumps:
\begin{align*}
\left\{
\begin{array}{lcl}
dY(t) &=& -f(t, Y(t), Z(t), J(t,\cdot))dt-g(t, Y(t), Z(t), J(t,\cdot))\overleftarrow{d}B(t)+Z(t)dW(t)\\
  &  & +\int_{\mathcal{E}}J(t, e)\tilde{N}(dt,de) -h(t, Y(t))dA(t), \ \ t\in[0,T],\\
 Y(T)&=&\xi,
\end{array}
\right.
\end{align*}
where $B$ and $W$ are two mutually independent standard Brownian motion processes, $\tilde{N}$ is an independent compensated Poisson random measure and $A$ is an increasing process. Since the celebrated work on nonlinear backward stochastic differential equations (BSDEs) by Pardoux and Peng \cite{Pardoux1990}, the interest for BSDEs has increasing rapidly, due to the many important applications in mathematical finance (see, e.g., \cite{Karoui1997}), stochastic control (see,  e.g., \cite{Yong1999}) and PDEs, especially the nonlinear Feynman-Kac formula and connections with viscosity solutions of PDEs (see,  e.g., \cite{Pardoux1998}). In order to give a probabilistic representation for the solution of
a class of quasi-linear SPDEs, Pardoux and Peng \cite{Pardoux1994} introduced a new kind
of BSDEs called BDSDEs where the integrals with respect to $B$ is a \lq\lq backward
It\^{o} integral\rq\rq, and the integrals with respect to $W$ is a
standard forward It\^{o} integral. Later, for the sake of giving a probabilistic formula for solutions of semilinear PDEs with Neumann boundary conditions,  Pardoux and Zhang \cite{PardouxZhang1998} introduced a new class of BSDEs called GBSDEs which involves an integral with respect to a continuous increasing process. In this paper, we combine the above two types of results and relate a new class of BSDEs with jumps called GBDSDEs with general jumps.

Inspired by the results of Buckdahn and Ma \cite{Buckdahn2001} and Boufoussi et al. \cite{Boufoussi2007}, we then consider the viscosity solution of the integral-partial differential equation (IPDE) with nonlinear Neumann boundary condition:
\begin{align*}
\left\{\begin{array}{l}
du(t,x)+\left[\mathcal{L}u(t,x)+f\left(t,x,u(t,x), (\nabla u\sigma)(t,x), \mathcal{A}u(t,x)\right)\right]dt=0, \ \ \ (t,x)\in [0,T]\times \mho,\\
\frac{\partial u}{\partial \mathbf{n}}(t,x)+h(t,x,u(t,x))=0, \ \ \ (t,x)\in [0,T]\times \partial\mho,\\
u(T,x)=l(x), \ \ \ x\in \bar{\mho}.
\end{array}
\right.
\end{align*}
Under mild conditions we prove that the solution of the GBSDE with jumps provides a viscosity solution of the above equation.

Finally, we investigate the singular optimal control problem for a stochastic
system whose state process is governed by a GBDSDE with general jumps. By the method of convex variation and
duality technique, we establish a sufficient and necessary
stochastic maximum principle for the stochastic system. To the best of our knowledge, Cadenillas and Haussmann \cite{Cadenillas1994} were the first to prove a maximum principle for singular stochastic optimal control. Later, Bahlali and Mezerdi \cite{Bahlali2005} established a general stochastic maximum principle for singular control problems by using a spike variation method. {\O}ksendal and Sulem  \cite{Oksendal2012} considered singular stochastic control of It\^{o}-L\'{e}vy processes and gave stochastic maximum principles. In particular, singular control is often related to impulse control. Recently, Wu \cite{Wu2019} studied the optimal impulse control problem of a currency with exchange rate dynamics whose state follows a geometric L\'{e}vy process.

The main novelty of the present work is manifested in the following two
distinguishing features. The first one is that we extend the theory of stochastic viscosity solutions of SPDEs to the SIPDEs with nonlinear Neumann boundary conditions which can be considered as a generalization of results in Buckdahn and Ma \cite{Buckdahn2001} and Boufoussi et al. \cite{Boufoussi2007}. More  precisely, we give some direct links between the stochastic viscosity solutions of  SIPDEs with nonlinear Neumann boundary conditions and the solutions of GBDSDEs with jumps which could be viewed as an extension of the nonlinear Feynman-Kac formula to SIPDEs. We would like to point out that our result is not an immediate extension of \cite{Buckdahn2001} and  \cite{Boufoussi2007}, since the integro-differential terms make the analysis more challenging. The second feature is that we not only prove the
existence and uniqueness of the solutions to the GBDSDEs with general jumps  but also establish the necessary and sufficient maximum
principles for singular stochastic optimal control of the GBDSDEs with general jumps for the first time.

This article is organized as follows. In Section 2, we
clarify all the notations to be used throughout the paper and provide some propositions which shall be used later. Section 3 is dedicated to the existence and uniqueness result for solutions to GBDSDEs with general jumps and prove some comparison theorems. Section 4 is devoted to a probabilistic representation for viscosity solutions of semi-linear IPDEs with nonlinear Neumann boundary conditions. In Section 5, we introduce the definition of stochastic viscosity solutions to semi-linear SIPDEs with nonlinear Neumann boundary conditions and give some direct links between the stochastic viscosity solutions of SIPDEs with nonlinear Neumann boundary conditions and the solutions of the GBDSDEs with jumps. Finally in Section 6, we formulate the singular stochastic optimal control problem of the GBDSDEs with general jumps and establish sufficient and necessary stochastic maximum principles.

\section{Preliminaries }\label{sec2}
In this section, we give the basic notation which will be used
throughout the paper.

We fix $T>0$ as an arbitrarily finite horizon. Let $(\Omega, \mathcal{F},  \mathbb{P})$ be a probability space. Let
$W(t)$ and $B(t)$ be two mutually independent standard Brownian motion
processes defined on $(\Omega, \mathcal{F}, \mathbb{P})$ with values
in $\mathbb{R}^d$ and $\mathbb{R}^m$ respectively.  Let
$\eta(t)$ be a pure
jump L\'{e}vy martingale on the probability space
$(\Omega, \mathcal{F}, \mathbb{P})$.
$\tilde{N}(dt,de):=N(dt,de)-\nu(de)dt$ is
the compensated jump measure of $\eta$, where $e\in \mathcal{E}:=\mathbb{R}^l\backslash\{0\}$.
Here, $\nu(de)$ is the L\'{e}vy measure of $\eta$
and $N(dt,de)$ is the jump measure of $\eta$. We assume
that the two Brownian motion processes and the pure jump
L\'{e}vy martingales are stochastically independent under
$\mathbb{P}$ and $\mathbb{F}:=\{\mathcal{F}_t\}_{0\leq t\leq T}$, where
$\mathcal{F}_t=\mathcal{F}_t^W\vee \mathcal{F}_t^N \vee
\mathcal{F}_{t, T}^B$, here for any stochastic process $X$,
$\mathcal{F}_{s, t}^X=\sigma\{X(r)-X(s); s\leq r\leq t\}\vee
\mathcal{N}$, $\mathcal{F}_t^X=\mathcal{F}_{0, t}^X$, where
$\mathcal{N}$ is the class of $\mathbb{P}$-null sets of
$\mathcal{F}$. One should note that the family of $\sigma$-algebras $\mathbb{F}$ is neither increasing nor decreasing, and hence it is not a filtration.
Let $\{A(t), 0\leq t\leq T\}$ be a one-dimensional increasing process satisfying $A(0)=0$ and  $A(t)$ is $\mathcal{F}_t$ measurable  for a.e. $t\in[0,T]$. We denote by
$|M|:=(M\cdot M)^\frac{1}{2}=(\mbox{tr}(MM^T))^\frac{1}{2}$ the norm
of a vector or a matrix $M$ and $M^*$ is the transpose of $M$.


For a generic Euclidean space $\mathbb{E}$ (if other Euclidean spaces are needed, we shall label them as $\mathbb{E}_1$, $\mathbb{E}_1$, $\cdots$, etc.), we introduce the following  notations or spaces:

$\bullet$  \ \ $\mathbb{F}^B:=\{\mathcal{F}_{t,T}^B\}_{0\leq t\leq T}$, $\mathbb{F}^{W\vee N}:=\{\mathcal{F}_{t}^W\vee\mathcal{F}_{t}^N\}_{0\leq t\leq T}$.

$\bullet$  \ \  $\mathcal{T}^B(0,T)$, the set of $\mathbb{F}^B$-stopping times $\tau$ such that $0\leq\tau\leq T$, a.s.

$\bullet$  \ \  $L^2(\mathcal{F}_T)$, the space of
 $\mathcal{F}_T$-measurable random variables
$\xi$ such that $E[|\xi|^2]<\infty$.

$\bullet$  \ \ $L^2(\mathbb{F},[0,T]; \mathbb{E})$: the space of $\mathbb{E}$-valued,
jointly measurable processes $\{X(t), t\in[0,
T]\}$ such that $E\left[\int_0^T|X(t)|^2dt\right]<\infty$ and $X(t)$ is $\mathcal{F}_t$-measurable for a.e. $t\in[0,T]$.

$\bullet$  \ \ $L^2(\mathbb{F},[0, T]; \nu, \mathbb{E})$,
the space of $\mathbb{E}$-valued,
jointly measurable processes
$\{X(t,e), (t,e)\in[0,T]\times \mathcal{E}\}$ such that
$E\left[\int_0^T\|X(t,\cdot)\|_\nu dt\right]<\infty$,
where $\|X(t,\cdot)\|_\nu=\left(\int_{\mathcal{E}}|X(t,e)|^2\nu(de)\right)^\frac{1}{2}$ and $X(t,e)$ is $\mathcal{F}_t$-measurable for a.e. $t\in[0,T]$ and $e\in\mathcal{E}$.

$\bullet$  \ \ $\mathcal{S}^2(\mathbb{F},[0, T]; \mathbb{E})$, the
space of $\mathbb{E}$-valued, c\`{a}dl\`{a}g
processes $\{X(t), t\in[0, T]\}$ such that
$\mathbb{E}\left[\sup\limits_{0\leq t\leq T}|X(t)|^2\right]<\infty$ and $X(t)$ is $\mathcal{F}_t$-measurable for a.e. $t\in[0,T]$.


$\bullet$  \ \  $L_\mu^2(A, \mathbb{F},[0,T];\mathbb{E})$,  the space of $\mathbb{E}$-valued,
jointly measurable processes $\{X(t), t\in[0,
T]\}$ such that $E\left[\int_0^Te^{\mu A(t)} |X(t)|^2dt+ \int_0^Te^{\mu A(t)}|X(t)|^2dA(t)\right]<\infty$ and $X(t)$ is $\mathcal{F}_t$-measurable for a.e. $t\in[0,T]$,
where $\mu\geq0$ and $L_{\mu}^2(\mathbb{F},[0,T];  \mathbb{E})$ the space of $\mathbb{E}$-valued, jointly measurable processes $\{X(t), t\in[0,T]\}$ such that $E\left[\int_0^Te^{\mu A(t)}|X(t)|^2dt\right]<\infty$, and $L_{\mu}^2(\mathbb{F},[0,T]; \nu, \mathbb{E})$,  the space of $\mathbb{E}$-valued, jointly measurable processes
$\{X(t,e), (t,e)\in[0,T]\times \mathcal{E}\}$ such that
$E\left[\int_0^T\int_{\mathcal{E}}e^{\mu A(t)}|X(t,e)|^2\nu(de)dt\right]<\infty$.

$\bullet$  \ \ $L^2(A, \mathbb{F},[0,T]; \mathbb{E})$: the space of $\mathbb{E}$-valued,
jointly measurable processes $\{X(t), t\in[0,
T]\}$ such that $E\left[\int_0^T|X(t)|^2dA(t)\right]<\infty$ and $X(t)$ is $\mathcal{F}_t$-measurable for a.e. $t\in[0,T]$.

$\bullet$  \ \  $V^2(\mathbb{F},[0,T]; \mathbb{E}_1\times \mathbb{E}_2\times \mathbb{E}_3):=\mathcal{S}^2(\mathbb{F},[0, T]; \mathbb{E}_1)\times L^2(\mathbb{F},[0,T]; \mathbb{E}_2)
\times L^2(\mathbb{F},[0, T];\nu, \mathbb{E}_3)$.

$\bullet$  \ \  $V_\mu^2(A, \mathbb{F},[0,T]; \mathbb{E}_1\times \mathbb{E}_2\times \mathbb{E}_3):=L_\mu^2(A, \mathbb{F},[0,T]; \mathbb{E}_1)\times L_{\mu}^2(\mathbb{F},[0,T]; \mathbb{E}_2)
\times L_{\mu}^2(\mathbb{F},[0,T]; \nu, \mathbb{E}_3)$.

$\bullet$  \ \  $C^{k,l}([0,T]\times\mathbb{E}; \mathbb{E}_1)$, the space of all $\mathbb{E}_1$-valued functions defined on $[0,T]\times\mathbb{E}$ which are $k$-times continuously differentiable in $t\in[0,T]$ and $l$-times continuously differentiable in $x\in\mathbb{E}$, and $C_b^{k,l}([0,T]\times\mathbb{E}; \mathbb{E}_1)$ denotes the subspace of $C^{k,l}([0,T]\times\mathbb{E}; \mathbb{E}_1)$ in which all functions have uniformly bounded partial derivatives.

$\bullet$  \ \ $\mathcal{C}^{k,l}(\mathcal{G},[0,T]\times\mathbb{E};\mathbb{E}_1)$ (resp. $\mathcal{C}_b^{k,l}(\mathcal{G},[0,T]\times\mathbb{E};\mathbb{E}_1)$), the space of $C^{k,l}(\mathcal{G},[0,T]\times\mathbb{E}; \mathbb{E}_1)$ (resp.  $C_b^{k,l}(\mathcal{G},[0,T]\times\mathbb{E}; \mathbb{E}_1)$ )-valued random variables that are $\mathcal{G}\otimes\mathcal{B}([0,T]\times\mathbb{R})$-measurable, where the sub-$\sigma$-field $\mathcal{G}\subseteq \mathcal{F}_T^B$.

$\bullet$  \ \ $\mathcal{C}^{k,l}(\mathbb{F}^B, [0,T]\times\mathbb{E}; \mathbb{E}_1)$ (resp. $\mathcal{C}_b^{k,l}(\mathbb{F}^B, [0,T]\times\mathbb{E}; \mathbb{E}_1)$), the space of $\mathcal{C}^{k,l}(\mathcal{F}_T^B, [0,T]\times\mathbb{E}; \mathbb{E}_1)$ (resp. $\mathcal{C}_b^{k,l}(\mathcal{F}_T^B, [0,T]\times\mathbb{E}; \mathbb{E}_1)$)-valued random fields $u$ such that for fixed $x\in\mathbb{E}$, the mapping $(\omega,t)\mapsto u(\omega,t,x)$ is $\mathbb{F}^B$-progressively measurable.

$\bullet$  \ \ For simplification of notations, when $\mathbb{E}=\mathbb{R}$, we shall omit $\mathbb{E}$ and just write $L^2(\mathbb{F},[0,T])$, etc, and denote $C([0,T]\times\mathbb{E}; \mathbb{E}_1)=C^{0,0}([0,T]\times\mathbb{E}; \mathbb{E}_1)$, etc.

We will use the following extension of It\^{o}'s formula which generalizes that in Pardoux and Peng \cite{Pardoux1994}.
\begin{proposition}\label{Ito}
let $\alpha\in \mathcal{S}^2(\mathbb{F},[0,T])$, $\beta\in
L^2(\mathbb{F},[0, T])$, $\gamma\in L^2(\mathbb{F},[0,T]; \mathbb{R}^m)$,
 $\delta\in L^2(\mathbb{F},[0,T]; \mathbb{R}^d)$, $\theta\in L^2(\mathbb{F},[0,T]; \nu)$, $\lambda\in L^2(A,\mathbb{F},[0,T])$ be such that
\begin{align*}
 \alpha(t) = &
\alpha(0)+\int_0^t\beta(s)ds +\int_0^t\gamma(s)\overleftarrow{d}B(s)+\int_0^t\delta(s)dW(s)+\\
&\int_0^t\int_{\mathcal{E}}\theta(s,e)\tilde{N}(ds,de)+\int_0^t\lambda(s)dA(s), \ \ t\in[0,T],
\end{align*}
then
\begin{align*}
\alpha^2(t)=&\alpha^2(0)+2\int_0^t\alpha(s)\beta(s)ds+2\int_0^t\alpha(s)
\gamma(s)\overleftarrow{d}B(s)+2\int_0^t\alpha(s)\delta(s)dW(s)\\
&+2\int_0^t\int_{\mathcal{E}}\alpha(s)\theta(s,
e)\tilde{N}(ds,de)+2\int_0^t\alpha(s)\lambda(s)dA(s)\\
&-\int_0^t|\gamma(s)|^2ds+\int_0^t|\delta(s)|^2ds+\int_0^t\int_{\mathcal{E}}\theta^2(s,e)N(ds,de)\\
&+\sum_{0<s\leq t}\lambda^2(s)(\Delta A(s))^2+\sum_{0<s\leq t}\int_{\mathcal{E}}\theta(s,e)N(\{s\},de)\lambda(s)\Delta A(s),
\end{align*}
\begin{align*}
E[\alpha^2(t)]=&E[\alpha^2(0)]+2E\left[\int_0^t\alpha(s)\beta(s)ds\right]+2E\left[\int_0^t\alpha(s)\lambda(s)dA(s)\right] -E\left[\int_0^t|\gamma(s)|^2ds\right]\\
&+E\left[\int_0^t|\delta(s)|^2ds\right]+E\left[\int_0^t\int_{\mathcal{E}}\theta^2(s, e)\nu(de)ds\right]\\
&+E\left[\sum_{0<s\leq t}\lambda^2(s)(\Delta A(s))^2 \right]+E\left[ \sum_{0<s\leq t}\int_{\mathcal{E}}\theta(s,e)N(\{s\},de)\lambda(s)\Delta A(s)\right],
\end{align*}
where $\Delta A(s)=A(s)-A(s-)$ and
\begin{align*}
\int_{\mathcal{E}}\theta(s,e)\tilde{N}(\{s\},de)=\left\{\begin{array}{l}
\theta(s,e), \mbox{if}\ \eta \ \mbox{has a jump of size} \ e \ \mbox{at} \ s,\\
0, \mbox{otherwise}.
\end{array}
\right.
\end{align*}

\end{proposition}
\begin{proof}
The proof can be performed similarly as in that of Lemma 1.3 in
Pardoux and Peng \cite{Pardoux1994}. So we do not repeat it here.
\end{proof}

Moreover, we will use the following jump-diffusion version of the It\^{o}-Ventzell formula which extends that in Buckdahn and Ma \cite{BuckdahnMa2001}.
\begin{proposition}\label{ItoVentzell}
Suppose that $F\in \mathcal{C}^{0,2}(\mathbb{F},[0,T]\times \mathbb{R}^n)$ is a semimartingale and $A$ is an adapted continuous increasing process. Let $G\in \mathcal{C}^{0,2}(\mathbb{F}^B,[0,T]\times \mathbb{R}^n)$,  $H\in \mathcal{C}^{0,2}(\mathbb{F}^B, [0,T]\times \mathbb{R}^n; \mathbb{R}^m)$ and $I\in \mathcal{C}^{0,2}(\mathbb{F}, [0,T]\times \mathbb{R}^n; \mathbb{R}^d)$ such that for any spatial parameter $x\in\mathbb{R}^n$,
\begin{align*}
F(t,x) = &
F(0,x)+\int_0^tG(s,x)ds +\int_0^tH(s,x)\overleftarrow{d}B(s)+\int_0^tI(s,x)dW(s), \ \ t\in[0,T].
\end{align*}
Let
 $\alpha\in \mathcal{S}^2(\mathbb{F},[0,T]; \mathbb{R}^n)$, $\beta\in
L^2(\mathbb{F}, [0, T]; \mathbb{R}^n)$, $\gamma\in L^2(\mathbb{F}, [0, T]; \mathbb{R}^{n\times m})$,
 $\delta\in L^2(\mathbb{F}, [0, T]; \mathbb{R}^{n\times d})$, $\theta\in L^2(\mathbb{F},[0,T]; \nu, \mathbb{R}^n)$, $\lambda\in L^2(A,\mathbb{F},[0,T]; \mathbb{R}^n)$ be such that
\begin{align*}
 \alpha(t) = &
\alpha(0)+\int_0^t\beta(s)ds +\int_0^t\gamma(s)\overleftarrow{d}B(s)+\int_0^t\delta(s)dW(s)\\
&+\int_0^t\int_{\mathcal{E}}\theta(s,e)\tilde{N}(ds,de)+\int_0^t\lambda(s)dA(s), \ \ t\in[0,T].
\end{align*}
Then
\begin{align*}
F(t,\alpha(t))=&F(0,\alpha(0))+\int_0^t G(s,\alpha(s))ds+\int_0^tH(s,\alpha(s))\overleftarrow{d}B(s)+\int_0^tI(s,\alpha(s))dW(s)\\
&+\int_0^t\nabla F(s,\alpha(s))\beta(s)ds +\int_0^t\nabla F(s,\alpha(s))\gamma(s)\overleftarrow{d}B(s)+\int_0^t\nabla F(s,\alpha(s))\delta(s)dW(s)\\
&+\int_0^t\int_{\mathcal{E}}\nabla F(s,\alpha(s))\theta(s,e)\tilde{N}(ds,de)+\int_0^t\nabla F(s,\alpha(s))\lambda(s)dA(s)\\
&+\int_0^t\int_{\mathcal{E}}\left\{F(s,\alpha(s-)+\theta(s,e))-F(s,\alpha(s-))-\nabla F(s,\alpha(s-))\theta(s,e)\right\}N(ds,de)\\
&+\frac{1}{2}\int_0^t\mbox{tr}\{\nabla^2F(s,\alpha(s))\delta(s)\delta^*(s)\}ds-\frac{1}{2}\int_0^t\mbox{tr}\{\nabla^2F(s,\alpha(s))\gamma(s)\gamma^*(s)\}ds\\
&+\int_0^t\mbox{tr}\{\nabla I(s,\alpha(s))\delta^*(s)\}ds-\int_0^t\mbox{tr}\{\nabla H(s,\alpha(s))\gamma^*(s)\}ds.
\end{align*}
\end{proposition}
\begin{proof}
The proof is analogous to  that of Theorem 4.2 in
Buckdahn and Ma \cite{BuckdahnMa2001}. So we omit the details here.
\end{proof}

If not specified, we will generically denote by $C$ the positive
constants that appear in the article below and may differ from line to
line.
\section{GBDSDEs with general jumps}\label{sec3} In this section, we study the existence and
uniqueness of the solution to the GBDSDEs with general jumps  and give the explicite solution of the linear GBDSDEs.

We consider the GBDSDE with general jumps:
\begin{align}\label{GBDSDE}
\left\{
\begin{array}{lcl}
dY(t) &=& -f(t, \Lambda(t))dt-g(t, \Lambda(t))\overleftarrow{d}B(t)+Z(t)dW(t)\\
  &  & +\int_{\mathcal{E}}J(t, e)\tilde{N}(dt,de) -h(t, Y(t))dA(t), \ \ t\in[0,T],\\
 Y(T)&=&\xi,
\end{array}
\right.
\end{align}
where
$$\Lambda(t)=(Y(t), Z(t), J(t, \cdot))$$ and
$$f: \Omega\times[0,T]\times \mathbb{R}\times \mathbb{R}^d\times L^2(\mathcal{E}, \nu)\to
 \mathbb{R},$$
$$g: \Omega\times[0,T]\times \mathbb{R}\times \mathbb{R}^d\times L^2(\mathcal{E}, \nu)\to
 \mathbb{R}^m,$$
$$h: \Omega\times[0,T]\times \mathbb{R}\to
\mathbb{ R},$$
are jointly measurable processes. Here the process $A(t)$  is a c\`{a}dl\`{a}g, increasing, $\mathcal{F}_t$-measurable process with $A(0)=0$  and $L^2(\mathcal{E}, \nu)$ is the set of functions $j: \mathcal{E}\to \mathbb{R}$ such that $\|j\|_\nu:=\left(\int_{\mathcal{E}}j^2(e)\nu(de)\right)^\frac{1}{2}<\infty$.

\begin{remark}
The integral with respect to $\{B(t)\}$ is a backward It\^{o} integral, in which the integrand takes values at the
 right end points of the subintervals in the Riemann type sum, while the integral with respect to $\{W(t)\}$ is a standard forward It\^{o} integral.
These two types of integrals are particular cases of the It\^{o}-Skorohod integral (see Nualart and Pardoux \cite{Nualart1988} for details).
In addition, we note that the process $Y$ has two types of jumps: inaccessible jumps, which come from the Poisson jumps $\left\{\int_0^t\int_{\mathcal{E}}J(s, e)N(ds,de), t\in[0,T]\right\}$, and predictable jumps, which stem from the predictable positive jumps of $A$.
\end{remark}

We assume that\\
(A3.1) the terminal condition $\xi\in L^2(\mathcal{F}_T)$ and for all $\mu>0$, $E[e^{\mu A_T}|\xi|^2]<\infty$.\\
 (A3.2) there exist a constant $ C>0$ and some $\mathcal{F}_t$-measurable processes $\{f_0(t), g_0(t), h_0(t); 0\leq t\leq T\}$ with values in $[1,+\infty)$,  and for any $(t,y,z,j)\in [0,T]\times\mathbb{R}\times \mathbb{R}^d\times L^2(\mathcal{E}, \nu)$, and $\mu>0$,
 $$|f(t,y,z,j)|\leq f_0(t)+C(|y|+|z|+\|j\|_\nu),$$
 $$|g(t,y,z,j)|\leq g_0(t)+C(|y|+|z|+\|j\|_\nu),$$
 $$|h(t,y)|\leq h_0(t)+C|y|,$$
 $$ E\left[\int_0^Te^{\mu A(t)}f_0^2(t)dt+\int_0^Te^{\mu A(t)}g_0^2(t)dt+\int_0^Te^{\mu A(t)}h_0^2(t)dA(t)\right]<\infty.$$
 (A3.3) there exist constants $ C>0$ and $0<\alpha<\frac{3}{4}$ such that $\forall t\in[0,T]$, $\forall y_1, y_2 \in \mathbb{R}, z_1, z_2, \in \mathbb{R}^d$ and $j_1, j_2\in L^2(\mathcal{E}, \nu)$,
 $$|f(t,y_1,z_1,j_1)-f(t,y_2,z_2,j_2)|^2\leq C(|y_1-y_2|^2+|z_1-z_2|^2+\|j_1-j_2\|_\nu^2),$$
 $$|g(t,y_1,z_1,j_1)-g(t,y_2,z_2,j_2)|^2\leq C(|y_1-y_2|^2)+\alpha(|z_1-z_2|^2+\|j_1-j_2\|_\nu^2),$$
 $$|h(t,y_1)-h(t,y_2)|^2\leq C(|y_1-y_2|^2).$$

\begin{remark}
In case the process $A(t)$ is a continuous increasing $\mathcal{F}_t$-measurable process, we may suppose that $\alpha$ satisfies the condition $0<\alpha<1$.
\end{remark}

\begin{definition}\label{def}
A solution of the GBDSDE (\ref{GBDSDE}) is a triple $(Y, Z, J)\in V^2(\mathbb{F},[0,T]; \mathbb{R}\times\mathbb{R}^d\times L^2(\mathcal{E}, \nu))$ such that
\begin{align*}
Y(t) =& \xi+\int_t^Tf(s, \Lambda(s))ds+\int_t^Tg(s, \Lambda(s))\overleftarrow{d}B(s)-\int_t^TZ(s)dW(s)\notag\\
  & -\int_t^T\int_{\mathcal{E}}J(s, e)\tilde{N}(ds,de) +\int_t^Th(s, Y(s))dA(s), \ \ 0\leq t\leq T.
\end{align*}
\end{definition}

\begin{theorem}\label{EXUNI}
Under assumptions (A3.1)-(A3.3), there exists a unique solution to the GBDSDE (\ref{GBDSDE}).
\end{theorem}
\begin{proof}
We shall prove the existence based on the fixed-point
theorem and divide the proof into two steps.\\
Step 1: Let $f$ and $g$  do not depend on $(y,z,j)$, i.e., $f(\omega,t,y,z,j)\equiv f(\omega,t)$,
$g(\omega,t,y,z,j)\equiv g(\omega,t)$, $\mathbb{P}-a.s$. We will prove
the following GBDSDE has a unique solution $(Y, Z, J)\in V^2(\mathbb{F},[0,T]; \mathbb{R}\times\mathbb{R}^d\times L^2(\mathcal{E}, \nu))$:
\begin{align}
Y(t)  =& \xi+ \int_t^Tf(s)ds-\int_t^TZ(s)dW(s) +\int_t^Tg(s)\overleftarrow{d}B(s)-
\int_t^T\int_{\mathcal{E}}J(s,e)\tilde{N}(ds,de)\notag\\
&+\int_t^Th(s)dA(s), \ \ t\in[0,T],
\end{align}
We define the filtration $\mathbb{G}=(\mathcal{G}_t)_{0\leq t\leq T}$ by
$$\mathcal{G}_t=\mathcal{F}_t^W\vee \mathcal{F}_t^N \vee
\mathcal{F}_{T}^B$$ and the $\mathcal{G}_t$-martingale by
$$M(t)=E\left[\xi+\int_0^Tf(s)ds+\int_0^Tg(s)\overleftarrow{d}B(s)
+\int_0^Th(s)dA(s)\Big|\mathcal{G}_t\right].$$
From (A3.1)-(A3.2), we can easily get that $M$ is
a square integrable $\mathcal{G}_t$-martingale. By virtue of the martingale representation theorem for jump diffusion processes
(see L{\o}kka, \cite{Lokka2004}), we can find $Z\in L^2(\mathbb{G},[0,T]; \mathbb{R}^d)$ and $J\in L^2(\mathbb{G},[0,T]; \nu)$ such that
$$M(t)=M(0)+\int_0^tZ(s)dW(s)+\int_0^t\int_{\mathcal{E}}J(s,e)\tilde{N}(ds,de).$$
Thus
$$M(T)=M(t)+\int_t^TZ(s)dW(s)+\int_t^T\int_{\mathcal{E}}J(s,e)\tilde{N}(ds,de).$$
Let
\begin{align*} Y(t)&=E\left[\xi+\int_t^Tf(s)ds+\int_t^Tg(s)\overleftarrow{d}B(s)
+\int_t^Th(s)dA(s)\Big|\mathcal{G}_t\right]\\&:=E[\Gamma|\mathcal{G}_t].
\end{align*}
By the defining formulas of $M(T)$ and $M(t)$, we get
\begin{align*}
Y(t)= &\xi+\int_t^Tf(s)ds-\int_t^TZ(s)dW(s) +\int_t^Tg(s)\overleftarrow{d}B(s)\\
&-
\int_t^T\int_{\mathcal{E}}J(s,e)\tilde{N}(ds,de)
 +\int_t^Th(s)dA(s).
\end{align*}
Next, we show $Y(t)$, $Z(t)$ and $J(t,\cdot)$ are in fact
$\mathcal{F}_t$-measurable. We note that $\mathcal{G}_t=\mathcal{F}_t\vee
\mathcal{F}_t^B$ and $\Gamma$ is
$\mathcal{F}_T^W\vee \mathcal{F}_T^N \vee \mathcal{F}_{t,T}^B$ measurable.
Using the fact that $\mathcal{F}_{t,T}^W$ and $\mathcal{F}_t$ are independent yields
 $Y(t)=E[\Gamma|\mathcal{F}_t]$. The proof for $Z(t)$ and $J(t,\cdot)$ are similar to that of Pardoux and Peng \cite{Pardoux1994}.\\
Step 2: For any $\theta=(y, z, j)\in V_\mu^2(A, \mathbb{F},[0,T]; \mathbb{R}\times\mathbb{R}^d\times L^2(\mathcal{E}, \nu))$, from Step 1,
we construct a mapping $\Phi$ from $V_\mu^2(A, \mathbb{F},[0,T]; \mathbb{R}\times\mathbb{R}^d\times L^2(\mathcal{E}, \nu))$ into itself by $\Phi(y,z,j)=(Y,Z,J)$ via the following
GBDSDE:
\begin{align*}
Y(t) =& \xi+\int_t^Tf(s, \theta(s))ds+\int_t^Tg(s, \theta(s))\overleftarrow{d}B(s)-\int_t^TZ(s)dW(s)\notag\\
  & -\int_t^T\int_{\mathcal{E}}J(s, e)\tilde{N}(ds,de) +\int_t^Th(s,y(s))dA(s), \ \ 0\leq t\leq T.
\end{align*}
 For any $\theta_1=(y_1,z_1,j_1)$,
$\theta_2=(y_2,z_2,j_2)\in V_\mu^2(A, \mathbb{F},[0,T]; \mathbb{R}\times\mathbb{R}^d\times L^2(\mathcal{E}, \nu))$, we set
$$(\bar{y},
\bar{z}, \bar{j})=(y_1-y_2,
z_1-z_2, j_1-j_2),$$
$$\bar{f}(t)=\mu(t, \theta_1(t))-\mu(t, \theta_2(t)), \bar{g}(t)=g(t, \theta_1(t))-g(t,
\theta_2(t)),  \bar{h}(t)=h(t, y_1(t))-h(t,
y_2(t)).$$
By Doob's inequality, we note that whenever $\theta=(y, z, j)\in V_\mu^2(A, \mathbb{F},[0,T]; \mathbb{R}\times\mathbb{R}^d\times L^2(\mathcal{E}, \nu))$,
\begin{align*}
(Y,Z,J)\in V_\mu^2(A, \mathbb{F},[0,T]; \mathbb{R}\times\mathbb{R}^d\times L^2(\mathcal{E}, \nu)) \ \ \mbox{and} \ \  E\left[\sup_{0\leq t\leq T}e^{\lambda t+\mu A(t)}Y^2(t)\right]<\infty.
\end{align*}
So that,  the processes $\int_0^te^{\lambda s+\mu A(s)}Y(s)Z(s)dW(s)$,  $\int_0^t\int_{\mathcal{E}}e^{\lambda s+\mu A(s)}Y(s)J(s,e)\tilde{N}(ds,de)$ and  $\int_0^te^{\lambda s+\mu A(s)}Y(s)g(s,\theta(s))\overleftarrow{d}B(s)$ are uniformly integrable martingales.
Let
$\lambda\in \mathbb{R}$ and $\mu>0$  be undetermined. By It\^{o}'s formula, we obtain
\begin{align*}
&E\left[\int_t^Te^{\lambda s+\mu A(s)}\left\{(\lambda\bar{Y}^2(s)+\bar{Z}^2(s)+\|\bar{J}(s,\cdot)\|_\nu^2)ds+\mu\bar{Y}^2(s)dA(s)\right\}\right]\\
&\ \ \ +E[\bar{Y}^2(t)e^{\lambda t+\mu A(t)}]+E\left[\sum_{t<s\leq T}e^{\lambda s+\mu A(s)}\bar{h}^2(s)(\Delta A(s))^2\right]\\
&=E\left[\int_t^T2e^{\lambda s+\mu A(s)}\bar{Y}(s)\bar{f}(s)ds\right]+E\left[\int_t^Te^{\lambda s+\mu A(s)}
\bar{g}^2(s)ds\right]\\
& \ \ \ + E\left[\int_t^T2e^{\lambda s+\mu A(s)}\bar{Y}(s)\bar{h}(s)dA(s)\right]+E\left[\sum_{t<s\leq T}e^{\lambda s+\mu A(s)}\int_{\mathcal{E}}
\bar{J}(s,e)N(\{s\},de)\bar{h}(s)\Delta A(s)\right]\\
&\leq \epsilon E\left[\int_t^T
e^{\lambda s+\mu A(s)}\bar{Y}^2(s)ds\right]+\left(\frac{C}{\epsilon}+C\right)E\left[\int_t^T
e^{\lambda s+\mu A(s)}\bar{y}^2(s)ds\right]\\
&\ \ \ +\left(\frac{C}{\epsilon}+\alpha\right)E\left[\int_t^T
e^{\lambda s+\mu A(s)}(\bar{y}^2(s)+\|\bar{z}(s,\cdot)|_\nu^2)ds\right]\\
&\ \ \ +E\left[\int_t^T
e^{\lambda s+\mu A(s)}\bar{Y}^2(s)dA(s)\right]+CE\left[\int_t^T
e^{\lambda s+\mu A(s)}\bar{y}^2(s)dA(s)\right]\\
& \ \  \  +\frac{1}{4}\int_t^Te^{\lambda s+\mu A(s)}\|\bar{J}(s,\cdot)\|_\nu^2ds +E\left[\sum_{t<s\leq T}e^{\lambda s+\mu A(s)}\bar{h}^2(s)(\Delta A(s))^2\right]
\end{align*}
Rearranging the terms give
\begin{align*}
&\left(\lambda-\epsilon\right)\frac{4}{3}E\left[\int_0^Te^{\lambda s+\mu A(s)}\bar{Y}^2(s)ds\right]+E\left[\int_0^Te^{\lambda s+\mu A(s)}\bar{Z}^2(s)ds\right]\\
& \ \ \  +E\left[\int_0^Te^{\lambda s+\mu A(s)}\|\bar{J}(s,
\cdot)\|_\nu^2ds\right]+\left(\mu-1\right)\frac{4}{3}E\left[\int_0^Te^{\lambda s+\mu A(s)}\bar{Y}^2(s)dA(s)\right]\\
&\leq \frac{4}{3}(\frac{C}{\epsilon}+\alpha)\left\{\frac{\frac{C}{\epsilon}+C}{\frac{C}{\epsilon}+\alpha}E\left[\int_0^T e^{\lambda s+\mu A(s)}\bar{y}^2(s)ds\right]+E\left[\int_0^T
e^{\lambda s+\mu A(s)}(\bar{z}^2(s)+\|\bar{j}(s,\cdot)\|_\nu^2)ds\right]\right.\\
& \ \ \ \left.+ \frac{C}{\frac{C}{\epsilon}+\alpha} E\left[\int_0^T e^{\lambda s+\mu A(s)}\bar{y}^2(s)dA(s)\right]\right\}.
\end{align*}
Choosing $\epsilon=\frac{C}{\frac{3}{8}-\frac{\alpha}{2}}$, $\lambda=\epsilon+\frac{3}{4}\frac{\frac{C}{\epsilon}+C}{\frac{C}{\epsilon}+\alpha}$, $\mu=1+\frac{C}{\frac{C}{\epsilon}+\alpha}$ and denoting $\hat{\lambda}=\frac{\frac{C}{\epsilon}+C}{\frac{C}{\epsilon}+\alpha}$,
$\hat{\mu}=\frac{C}{\frac{C}{\epsilon}+\alpha}$,
 we arrive at
\begin{align*}
&E\left[\int_0^Te^{\lambda s+\mu A(s)}\hat{\lambda}\bar{Y}^2(s)ds\right]+E\left[\int_0^Te^{\lambda s+\mu A(s)}\bar{Z}^2(s)ds\right]\\
& \ \ +E\left[\int_0^Te^{\lambda s+\mu A(s)}\|\bar{J}(s,
\cdot)\|_\nu^2ds\right]+ E\left[\int_0^Te^{\lambda s+\mu A(s)}\hat{\mu}\bar{Y}^2(s)dA(s)\right]\\
&\leq \left(\frac{1}{2}+\frac{2}{3}\alpha\right)\left\{E\left[\int_0^T e^{\lambda s+\mu A(s)}\hat{\lambda}\bar{y}^2(s)ds\right]+E\left[\int_0^T e^{\lambda s+\mu A(s)}\bar{z}^2(s)ds\right]\right.\\
&\ \  \left.+E\left[\int_0^T e^{\lambda s+\mu A(s)}\|\bar{j}(s,\cdot)\|_\nu^2ds\right]+ E\left[\int_0^Te^{\lambda s+\mu A(s)}\hat{\mu}\bar{y}^2(s)dA(s)\right]\right\}.
\end{align*}
Since $\frac{1}{2}+\frac{2}{3}\alpha<1$, $\Phi$ is a strict contractive
mapping on $V_\mu^2(A, \mathbb{F},[0,T]; \mathbb{R}\times\mathbb{R}^d\times L^2(\mathcal{E}, \nu))$ equipped with the norm
\begin{align*}
\parallel\mid(Y,Z,J)\mid\parallel_{\lambda, \hat{\lambda},\mu,\hat{\mu}}
=\left(E \int_0^T e^{\lambda s+\mu A(s)}\left\{\left[\hat{\lambda}Y^2(s)+Z^2(s)+\| J(s,\cdot)\|_\nu^2\right]ds+\hat{\mu}Y^2(s)dA(s)\right\}\right)^\frac{1}{2}.
\end{align*}
Hence, from the fixed-point
theorem, this mapping admits a fixed point
$\Theta=(Y, Z, J)$ such that
$\Phi(\Theta)=\Theta$, which is a unique
solution of (\ref{GBDSDE}).

By Burkholder-Davis-Gundy's inequality, Doob's inequality and
H\"{o}lder's inequality, we have
\begin{align*}
E\left[\sup_{t\in[0,T]}Y^2(t)\right]<\infty.
\end{align*}

Therefore, we obtain
$(Y, Z, J)\in V^2(\mathbb{F},[0,T]; \mathbb{R}\times\mathbb{R}^d\times L^2(\mathcal{E}, \nu))$. The proof is complete.
\end{proof}

\begin{remark}
If $h$ does not depend on $y$, i.e., $h(\omega,t,y)\equiv h(\omega,t)$, then we define $\tilde{Y}(t):=Y(t)+\int_0^th(s)dA(s)$, $\tilde{Z}(t):=Z(t)$, $\tilde{J}(t,\cdot):=J(t,\cdot)$. We notice that $(\tilde{Y}, \tilde{Z}, \tilde{J})$ solves the following BDSDE with jumps:
\begin{align}\label{BDSDEJ}
\left\{
\begin{array}{lcl}
d\tilde{Y}(t) &=& -\tilde{f}(t, \tilde{\Lambda}(t))dt-g(t, \tilde{\Lambda}(t))\overleftarrow{d}B(t)+\tilde{Z}(t)dW(t)\\
  &  & +\int_{\mathcal{E}}\tilde{J}(t, e)\tilde{N}(dt,de), \ \ t\in[0,T],\\
 \tilde{Y}(T)&=&\xi+\int_0^Th(s)dA(s),
\end{array}
\right.
\end{align}
where
$$\tilde{\Lambda}(t)=(\tilde{Y}(t), \tilde{Z}(t), \tilde{J}(t, \cdot)),$$
$$\tilde{f}(t, y,z,j):=f\left(t,y-\int_0^th(s)dA(s),z,j\right).$$
It's easy to check that $\tilde{f}$ is Lipschitz in $(y,z,j)$, $\tilde{f}(\cdot,0,0)\in L^2(\mathbb{F},[0,T])$ and $\xi+\int_0^Th(s)dA(s)\in L^2(\mathcal{F}_T)$. From arguments in Wu and Liu \cite{Wu2018}, the BDSDE with jumps (\ref{BDSDEJ}) admits a unique solution $(\tilde{Y}, \tilde{Z}, \tilde{J})\in V^2(\mathbb{F},[0,T]; \mathbb{R}\times\mathbb{R}^d\times L^2(\mathcal{E}, \nu))$.
Hence, the GBDSDE (\ref{GBDSDE}) has a unique solution $(Y, Z, J)\in V^2(\mathbb{F},[0,T]; \mathbb{R}\times\mathbb{R}^d\times L^2(\mathcal{E}, \nu))$.
\end{remark}

There is no solution formula for the general GBDSDE (\ref{GBDSDE}). However, in the linear case we get the following:
\begin{theorem}\label{TLGBDSDE}
Let $\alpha$, $\beta$, $\gamma$, $\delta$ be bounded $\mathcal{F}_t$-measurable processes, $\xi\in L^2(\mathcal{F}_T)$, $\phi\in L^2(\mathbb{F},[0,T])$,
$\varphi\in L^2(\mathbb{F},[0,T]; \mathbb{R}^m)$, $h\in L^2(A, \mathbb{F},[0,T])$. Assume $\gamma>-1$ a.s. Then the unique solution $(Y,Z,J)$ of the linear GBDSDE
\begin{align*}
\left\{
\begin{array}{lcl}
dY(t) &=& -[\phi(t)+\alpha(t)Y(t)+\beta(t)Z(t)+\int_{\mathcal{E}}\gamma(t,e)J(t,e)\nu(de)]dt\\
&  & -[\varphi(t)+\delta(t)Y(t)]\overleftarrow{d}B(t)+Z(t)dW(t)\\
&  &  +\int_{\mathcal{E}}J(t, e)\tilde{N}(dt,de) -h(t)dA(t), \ \ 0\leq t\leq T,\\
 Y(T)&=&\xi,
\end{array}
\right.
\end{align*}
is given by
\begin{align}\label{SLGBDSDE}
Y(t)=&E\left[\frac{\Gamma(T)}{\Gamma(t)}\xi+\int_t^T\frac{\Gamma(s)}{\Gamma(t)}(\phi(s)-\delta(s)\varphi(s))ds\right.\notag\\
&+\int_t^T\frac{\Gamma(s)}{\Gamma(t)}h(s)dA(s)+\int_t^T\frac{\Gamma(s)}{\Gamma(t)}\varphi(s)\overleftarrow{d}B(s)\notag\\
&\left. +\sum_{t<s\leq T}\frac{\Gamma(s)}{\Gamma(t)}\int_{\mathcal{E}}\gamma(s,e)N(\{s\},de)h(s)\Delta A(s)\Big|\mathcal{G}_t\right],  \ \ \ 0\leq t\leq T,
\end{align}
where
\begin{align*}
\left\{
\begin{array}{lcl}
d\Gamma(t) &=& \Gamma(t)\left[(\alpha(t)-\delta^2(t))dt+\beta(t)dW(t)\right.\\
& & \left.+\delta(t)\overleftarrow{d}B(t)+\int_{\mathcal{E}}\gamma(t,e)\tilde{N}(dt,de)\right], \ \ 0\leq t\leq T,\\
 \Gamma(0)&=&1,
\end{array}
\right.
\end{align*}
i.e.
\begin{align}\label{Gamma}
\Gamma(t)=&\exp\left(\int_0^t\left\{\alpha(s)-\frac{1}{2}\delta^2(s)-\frac{1}{2}\beta^2(s)\right\}ds+\int_0^t\beta(s)dW(s)
+\int_0^t\delta(s)\overleftarrow{d}B(s)\right.\notag\\
&\left. +\int_0^t\int_{\mathcal{E}}\ln (1+\gamma(s,e))\tilde{N}(ds,de)+\int_0^t\int_{\mathcal{E}}\{\ln(1+\gamma(s,e))-\gamma(s,e)\}\nu(de)ds\right).
\end{align}
\end{theorem}
\begin{proof}
By It\^{o}'s formula, we have
\begin{align*}
d\Gamma(t)Y(t)=&-\Gamma(t)\left[\phi(t)+\alpha(t)Y(t)+\beta(t)Z(t)+\int_{\mathcal{E}}\gamma(t,e)J(t,e)\nu(de)\right]dt\\
&+\Gamma(t)Z(t)dW(t)-
\Gamma(t)[\varphi(t)+\delta(t)Y(t)]\overleftarrow{d}B(t)+\Gamma(t)\int_{\mathcal{E}}J(t,e)\tilde{N}(dt,de)\\
&-\Gamma(t)h(t)dA(t)+
Y(t)\Gamma(t)[\alpha(t)-\delta^2(t)]dt+Y(t)\Gamma(t)\beta(t)dW(t)\\
&+Y(t)\Gamma(t)\delta(t)\overleftarrow{d}B(t)
+Y(t)\Gamma(t)\int_{\mathcal{E}}\gamma(t,e)\tilde{N}(dt,de)+Z(t)\Gamma(t)\beta(t)dt\\
&+\Gamma(t)[Y(t)\delta(t)+\varphi(t)]\delta(t)dt
+\Gamma(t)\int_{\mathcal{E}}J(t,e)\gamma(t,e)N(dt,de)\\
&-\Gamma(t)\int_{\mathcal{E}}\gamma(t,e)N(\{t\},de)h(t)\Delta A(t).
\end{align*}
Taking conditional expectation with respect to $\mathcal{G}_t$ yields
\begin{align*}
\Gamma(t)Y(t)=&E\left[\Gamma(T)Y(T)+\int_t^T\Gamma(s)(\phi(s)-\delta(s)\varphi(s))ds\right.\notag\\
&+\int_t^T\Gamma(s)h(s)dA(s)+\int_t^T\Gamma(s)\varphi(s)\overleftarrow{d}B(s)\notag\\
&\left. +\sum_{t<s\leq T}\Gamma(s)\int_{\mathcal{E}}\gamma(s,e)N(\{s\},de)h(s)\Delta A(s)\Big|\mathcal{G}_t\right].
\end{align*}
Hence, we get (\ref{SLGBDSDE}).
\end{proof}
\begin{remark}
Using Malliavin calculus, we can write
\begin{align*}
Z(t)=D_{t-}Y(t):=\lim_{s\to t-}D_s Y(t), \ \  \  J(t,e)=D_{t-,e}Y(t):=\lim_{s\to t-}D_{s,e} Y(t),
\end{align*}
where $D_s Y(t)$ denotes the Malliavin derivative of $Y(t)$ at $s$ with respect to Brownian motion and $D_{s,e}Y(t)$ denotes the Malliavin derivative of $Y(t)$ at $(s,e)$ with respect to the jump measure.
\end{remark}

We note that there is no efficient comparison theorem for solutions of standard GBDSDEs with general jumps. However, in some special case, we provide a comparison theorem under proper conditions.
\begin{lemma}\label{Co}
Let $\alpha, \beta, \gamma, \varphi, \delta, h, \xi$ be as in Theorem \ref{TLGBDSDE}. Suppose $(Y,Z,J)$ satisfies the following
linear GBDSDE with jumps:
\begin{align}\label{CLGBDSDE}
\left\{
\begin{array}{lcl}
dY(t) &=& -\chi(t)dt-[\varphi(t)+\delta(t)Y(t)]\overleftarrow{d}B(t)+Z(t)dW(t)\\
&  &  +\int_{\mathcal{E}}J(t, e)\tilde{N}(dt,de) -h(t)dA(t), \ \ 0\leq t\leq T,\\
 Y(T)&=&\xi,
\end{array}
\right.
\end{align}
where  $\chi(t)$ is a given $\mathcal{F}_t$-measurable process such that
\begin{align*}
\chi(t)\geq\alpha(t)Y(t)+\beta(t)Z(t)+\int_{\mathcal{E}}\gamma(t,e)J(t,e)\nu(de).
\end{align*}
Then
\begin{align}\label{Com}
Y(t)\geq&E\left[\frac{\Gamma(T)}{\Gamma(t)}\xi-\int_t^T\frac{\Gamma(s)}{\Gamma(t)}\delta(s)\varphi(s)ds+\int_t^T\frac{\Gamma(s)}{\Gamma(t)}h(s)dA(s)
+\int_t^T\frac{\Gamma(s)}{\Gamma(t)}\varphi(s)\overleftarrow{d}B(s)\right.\notag\\
&\left.+\sum_{t<s\leq T}\frac{\Gamma(s)}{\Gamma(t)}\int_{\mathcal{E}}\gamma(s,e)N(\{s\},de)h(s)\Delta A(s)\Big|\mathcal{G}_t\right],  \ \ \ 0\leq t\leq T,
\end{align}
where $\Gamma(\cdot)$ is given by (\ref{Gamma}).
\end{lemma}
\begin{proof}
Using It\^{o}'s formula yields
\begin{align*}
d\Gamma(t)Y(t)=&\Gamma(t)\left\{-\chi(t)dt+Z(t)dW(t)-[\varphi(t)+\delta(t)Y(t)]\overleftarrow{d}B(t)\right.\\
&\left.+\int_{\mathcal{E}}J(t,e)\tilde{N}(dt,de)-h(t)dA(t)\right\}\\
&+Y(t)\Gamma(t)\left\{[\alpha(t)-\delta^2(t)]dt+\beta(t)dW(t)+\delta(t)\overleftarrow{d}B(t)\right.\\
&\left.+\int_{\mathcal{E}}\gamma(t,e)\tilde{N}(dt,de)\right\}+\Gamma(t)\left\{[Z(t)\beta(t)+(Y(t)\delta(t)+\varphi(t))\delta(t)]dt\right.\\
&\left.+\int_{\mathcal{E}}J(t,e)\gamma(t,e)N(dt,de)-\int_{\mathcal{E}}\gamma(t,e)N(\{t\},de)h(t)\Delta A(t)\right\}\\
&\leq \Gamma(t)\left\{-\left[\alpha(t)Y(t)+\beta(t)Z(t)+\int_{\mathcal{E}}\gamma(t,e)J(t,e)\nu(de) \right]dt\right.\\
&\left.+Z(t)dW(t)-[\varphi(t)+\delta(t)Y(t)]\overleftarrow{d}B(t)+\int_{\mathcal{E}}J(t,e)\tilde{N}(dt,de)-h(t)dA(t)\right\}\\
&+Y(t)\Gamma(t)\left\{[\alpha(t)-\delta^2(t)]dt+\beta(t)dW(t)+\delta(t)\overleftarrow{d}B(t)\right.\\
&\left.+\int_{\mathcal{E}}\gamma(t,e)\tilde{N}(dt,de)\right\}+\Gamma(t)\left\{[Z(t)\beta(t)+(Y(t)\delta(t)+\varphi(t))\delta(t)]dt\right.\\
&\left.+\int_{\mathcal{E}}J(t,e)\gamma(t,e)N(dt,de)-\int_{\mathcal{E}}\gamma(t,e)N(\{t\},de)h(t)\Delta A(t)\right\}.
\end{align*}
Taking conditional expectation gives (\ref{Com}).
\end{proof}


\begin{theorem}
Suppose we have two process triples $(Y_1, Z_1, J_1)$ and $(Y_2, Z_2, J_2)$ satisfy
\begin{align*}
\left\{
\begin{array}{lcl}
dY_i(t) &=& -f_i(t, Y_i(t), Z_i(t), J_i(t,\cdot))dt-\delta(t)Y_i(t)\overleftarrow{d}B(t)+Z_i(t)dW(t)\\
&  &  +\int_{\mathcal{E}}J_i(t, e)\tilde{N}(dt,de) -h_i(t)dA(t), \ \ t\in[0,T],\\
 Y_i(T)&=&\xi_i,
\end{array}
\right. \ \ \  i=1,2,
\end{align*}
where $A$ is a continuous increasing process and $f_i(t, y, z, j): \Omega\times[0,T]\times\mathbb{R}\times \mathbb{R}^d\times L^2(\mathcal{E},\nu)\to\mathbb{R}$,
$h_i(t):  \Omega\times[0,T]\to\mathbb{R}$, $i=1,2$  satisfy (A3.2) and (A3.3) and $\delta(t):
\Omega\times[0,T]\to\mathbb{R}$ is a bounded  process and $\delta(t)$ is $\mathcal{F}_t$-measurable for a.e. $t\in[0,T]$.
Assume that there exists a bounded process $\theta$ independent of $y$ and $z$, $\theta(t,e)>-1$ and $\theta\in L^2(\mathcal{E},\nu)$ such that for all $t,y,z$,
\begin{align*}
f_2(t,y,z,j_2(\cdot))-f_2(t,y,z,j_1(\cdot))\geq \int_{\mathcal{E}}\theta(t,e)(j_2(e)-j_1(e))\nu(de).
\end{align*}
Suppose that
\begin{align*}
f_1(t,Y_1(t), Z_1(t), J_1(t,\cdot))\leq f_2(t,Y_1(t), Z_1(t), J_1(t,\cdot)),\ \ \ t\in[0,T], \ \ \ a.s.
\end{align*}
and $\xi_1\leq \xi_2$,  $0\leq h_1(t)\leq h_2(t)$, a.s. Then
$$Y_1(t)\leq Y_2(t), \ \ \ t\in[0,T], \ \ \ a.s.$$
\end{theorem}
\begin{proof}
Define
\begin{align*}
(\bar{Y}, \bar{Z}, \bar{J})=(Y_2-Y_1, Z_2-Z_1, J_2-J_1), \bar{h}=h_2-h_1, \bar{\xi}=\xi_2-\xi_1.
\end{align*}
Then
\begin{align*}
\left\{
\begin{array}{lcl}
d\bar{Y}(t) &=& -[f_2(t, \Lambda_2(t))-f_1(t,\Lambda_1(t))]dt-\delta(t)\bar{Y}(t)\overleftarrow{d}B(t)+\bar{Z}(t)dW(t)\\
&  &  +\int_{\mathcal{E}}\bar{J}(t, e)\tilde{N}(dt,de) -\bar{h}(t)dA(t), \ \ t\in[0,T],\\
 \bar{Y}(T)&=&\bar{\xi}.
\end{array}
\right.
\end{align*}
Denote
\begin{align*}
\chi(t):&=f_2(t, Y_2(t), Z_2(t), J_2(t,\cdot))-f_1(t, Y_1(t), Z_1(t), J_1(t,\cdot)), \\
\kappa(t): &=f_2(t, Y_1(t), Z_1(t), J_1(t,\cdot))-f_1(t, Y_1(t), Z_1(t), J_1(t,\cdot)), \\
\alpha(t):&=\frac{1}{\bar{Y}(t)}[f_2(t, Y_2(t), Z_2(t), J_2(t,\cdot))-f_2(t, Y_1(t), Z_2(t), J_2(t,\cdot))\mathbf{1}_{\bar{Y}(t)\neq0},\\
\beta(t):&=\frac{1}{\bar{Z}(t)}[f_2(t, Y_1(t), Z_2(t), J_2(t,\cdot))-f_2(t, Y_1(t), Z_1(t), J_2(t,\cdot))\mathbf{1}_{\bar{Z}(t)\neq0}.
\end{align*}
We note that the processes $\alpha$ and $\beta$ are jointly measurable bounded processes. Using the assumptions and combining the above notations, we get
\begin{align*}
\left\{
\begin{array}{lcl}
d\bar{Y}(t) &=& -\chi(t)dt-\delta(t)\bar{Y}(t)\overleftarrow{d}B(t)+\bar{Z}(t)dW(t)\\
&  &  +\int_{\mathcal{E}}\bar{J}(t, e)\tilde{N}(dt,de) -\bar{h}(t)dA(t), \ \ \ 0\leq t\leq T,\\
 \bar{Y}(T)&=&\bar{\xi},
\end{array}
\right.
\end{align*}
where
\begin{align*}
\chi(t)&=\kappa(t)+\alpha(t)\bar{Y}(t)+\beta(t)\bar{Z}(t)+\int_{\mathcal{E}}\theta(t,e)\bar{J}(t,e)\nu(de)\\
&\geq\alpha(t)\bar{Y}(t)+\beta(t)\bar{Z}(t)+\int_{\mathcal{E}}\theta(t,e)\bar{J}(t,e)\nu(de).
\end{align*}
From Lemma \ref{Co}, we have $\bar{Y}(t)\geq0$ for all $t$, i.e., $Y_2(t)\geq Y_1(t)$ for all $t$.
\end{proof}

\section{Viscosity solutions to IPDEs with nonlinear Neumann boundary conditions}\label{sec4}
In this section, we present a probabilistic formula for the solution of a parabolic IPDE with nonlinear Neumann boundary condition.

Given a function $\phi\in C_b^2(\mathbb{R}^n)$, let $\mho$ be an open connected bound subset of $\mathbb{R}^n$ such that
$\mho=\{\phi>0\}$, $\partial\mho=\{\phi=0\}$, and $|\nabla\phi(x)|=1$, $x\in\partial\mho$.
Consider the SDE with reflection
\begin{align}\label{RSDE}
\left\{\begin{array}{lcl}
X^{t,x}(s)&=&x+\int_t^{t\vee s} b(r, X^{t,x}(r))dr+\int_t^{t\vee s} \sigma(r, X^{t,x}(r))dW(r)\\
& &+\int_t^{t\vee s}\gamma(r, X^{t,x}(r), e)d\tilde{N}(dr,de) +\int_t^{t\vee s} \nabla\phi( X^{t,x}(r))dA^{t,x}(r), \\
A^{t,x}(s)&=&\int_t^{t\vee s}\mathbf{1}_{\{X^{t,x}(r)\in \partial\mho\}}dA^{t,x}(r), \ \ \  A \ \mbox{is increasing}, \ \ \ 0\leq s\leq T,
\end{array}
\right.
\end{align}
where $b: [0, T]\times \bar{\mho}\to \mathbb{R}^n$, $\sigma: [0, T]\times \bar{\mho}\to \mathbb{R}^{n\times d}$, $\gamma: [0, T]\times \bar{\mho}\times \mathcal{E}\to \mathbb{R}^n$
such that for some $C$, all $r\in[0,T], x, x'\in \mathbb{R}^n$:
\begin{align*}
&|b(r, x)|^2+|\sigma(r, x)|^2+\int_{\mathcal{E}}|\gamma(r, x,e)|^2\nu(de)\leq C,\\
&|b(r, x)- b(r, x')|^2+|\sigma(r, x)- \sigma(r, x')|^2+\int_{\mathcal{E}}|\gamma(r, x,e)- \gamma(r, x',e)|^2\nu(de)\leq C|x- x'|^2.
\end{align*}

In this section and the next section, we assume that $A$ is continuous and $x+\gamma(t,x,e)\in\bar{\mho}$, $\forall t\in[0,T]$, $\forall x\in\bar{\mho}$, $\forall e\in\mathcal{E}$. 
\begin{proposition}\label{pro3}
For any $x\in \bar{\mho}, t\in[0,T]$, there exists a unique pair of $\mathcal{F}_t^W\vee\mathcal{F}_t^N$-progressively measurable processes $\{(X^{t,x}(s), A^{t,x}(s)), s\geq 0\}$ with values
in $\bar{\mho}\times \mathbb{R}_+$ solving the equation (\ref{RSDE}) and for all $ x, x'\in \bar{\mho}$, there exists a positive constant $C$ such that
\begin{align*}
E\left[\sup_{0\leq s\leq T}|X^{t,x}(s)-X^{t,x'}(s)|^4\right]\leq C |x-x'|^4, \ \ \ E\left[\sup_{0\leq s\leq T}|A^{t,x}(s)-A^{t,x'}(s)|^4\right]\leq C |x-x'|^4.
\end{align*}
\end{proposition}
\begin{proof}
The existence and uniqueness of solution to the equation (\ref{RSDE}) follows from  Menaldi and Robin \cite{Menaldi1985}.  Adopting  the same method as in Lions and Sznitman \cite{Lions1984} (see also Pardoux and Zhang \cite{PardouxZhang1998}),  we can deduce the regularity results.
\end{proof}

Consider the GBSDE with jumps:
\begin{align}\label{VGBSDE}
\begin{array}{lcl}
Y^{t,x}(s)&=&l(X^{t,x}(T))+\int_{s}^Tf\left(r, X^{t,x}(r), Y^{t,x}(r), Z^{t,x}(r), \int_{\mathcal{E}}J^{t,x}(r,e)\nu(de)\right)dr\\
& & -\int_{s}^TZ^{t,x}(r)dW(r)- \int_{ s}^T\int_{\mathcal{E}}J^{t,x}(r,e)\tilde{N}(dr,de)\\
& & +\int_{ s}^Th(r,X^{t,x}(r), Y^{t,x}(r))dA^{t,x}(r), \ \ \ t\leq s\leq T,
\end{array}
\end{align}
where  $l: \bar{\mho}\to \mathbb{R}$,
$f: [0, T]\times \bar{\mho} \times \mathbb{R} \times \mathbb{R}^d \times \mathbb{R}\to \mathbb{R}$ and $h: [0, T]\times \bar{\mho}\times \mathbb{R}^d\to \mathbb{R}$ are continuous.

Assume that\\
(A4.1) for some $C>0$ and $C_0<0$, all $x, x'\in\bar{\mho}$, $y, y'\in \mathbb{R}$, $z, z'\in \mathbb{R}^d$, $j, j'\in \mathbb{R}$:
\begin{align*}
&|f(s,x,y,z,j)|^2\leq C(1+|x|^2+|y|^2+|z|^2+|j|^2), \\
&|h(s,x,y)|^2\leq C(1+|x|^2+|y|^2), \ \ \ |l(x)|^2\leq C(1+|x|^2),\\
&(y-y')(h(s, x, y)- h(s,x',y'))\leq C|x- x'|+C_0|y- y'|,\\
&|h(s, x, y)- h(s,x',y')|^2\leq C(|x- x'|^2+|y- y'|^2),\\
&|f(s, x, y,z,j)- f(s,x',y',z',j')|^2\leq C(|x- x'|^2+|y- y'|^2+|z- z'|^2+|j- j'|^2).
\end{align*}
(A4.2) for each $(t,x,y,z)\in[0,T]\times\bar{\mho}\times\mathbb{R}\times\mathbb{R}^d$, the function $j\mapsto f(t,x,y,z,j)$ is non-decreasing.

\begin{proposition}\label{continuous}
Under the assumption (A4.1), the GBSDE with jumps (\ref{VGBSDE}) has a unique solution $\{(Y^{t,x}(s), Z^{t,x}(s), J^{t,x}(s,\cdot)), t\leq s\leq T\}\in V^2(\mathbb{F}^{W\vee N},[t,T];\mathbb{R}\times \mathbb{R}^d\times L^2(\mathcal{E},\nu))$ and $(t,x)\mapsto Y^{t,x}(t)$ is a deterministic mapping from $[0,T]\times \bar{\mho}$. In addition,
 the deterministic mapping $(t,x)\mapsto Y^{t,x}(t)$ is continuous in $(t,x)$.
\end{proposition}
\begin{proof}
It follows from Theorem \ref{EXUNI} (when $g\equiv0$) that the GBSDE with jumps (\ref{VGBSDE}) has a unique solution $\{(Y^{t,x}(s), Z^{t,x}(s), J^{t,x}(s,\cdot)), t\leq s\leq T\}\in  V^2(\mathbb{F}^{W\vee N},[t,T];\mathbb{R}\times \mathbb{R}^d\times L^2(\mathcal{E},\nu))$. Clearly, for each $t\leq s\leq T$, $Y^{t,x}(s)$ is $\mathcal{F}^W_{t,s}\vee \mathcal{F}^N_{t,s}$ measurable. Thus $(t,x)\mapsto Y^{t,x}(t)$ is a deterministic mapping from $[0,T]\times \bar{\mho}$.

Let $(t,x)$ and $(t',x')$ be two elements of $[0,T]\times\bar{\mho}$.  Without loss of generality, we assume $A(T)$ is a bounded random variable.
For notational simplicity we define
\begin{align*}
\bar{Y}(s)&=Y^{t,x}(s)-Y^{t',x'}(s), \ \  \bar{Z}(s)=Z^{t,x}(s)-Z^{t',x'}(s), \\
 \ \bar{J}(s,\cdot)&=J^{t,x}(s,\cdot)-J^{t',x'}(s,\cdot), \ \ \bar{A}(s)=A^{t,x}(s)-A^{t',x'}(s), \ \  \tilde{A}(s)=\|\bar{A}\|(s)+A^{t',x'}(s),\\
 \bar{f}(s)&=f\left(s,X^{t,x}(s), Y^{t,x}(s), Z^{t,x}(s), \int_\mathcal{E}J^{t,x}(s,e)\nu(de)\right)\\
 & \ \ \  -f\left(s,X^{t',x'}(s), Y^{t',x'}(s), Z^{t',x'}(s), \int_\mathcal{E}J^{t',x'}(s,e)\nu(de)\right),\\
 \bar{h}(s)&=h(s,X^{t,x}(s), Y^{t,x}(s))-h(s,X^{t',x'}(s), Y^{t',x'}(s)).
\end{align*}
where $\|\bar{A}\|(s)$ is the total variation of the process $\bar{A}$ on the interval $[0,s]$.

Using It\^{o}'s formula, assumption (A4.1) and Young's inequality, we derive that for every $\mu>0$, $\delta_1>0$ and $\delta_2>0$,
\begin{align*}
E[e^{\mu\tilde{A}(s)}&\bar{Y}^2(s)]+E\left[\int_s^Te^{\mu\tilde{A}(r)}\{\bar{Z}^2(r)+\|\bar{J}(r,\cdot)\|_\nu^2\}dr
+\int_s^T\mu e^{\mu\tilde{A}(r)}\bar{Y}^2(r)d\tilde{A}(r)\right]\\
 =&E\left[e^{\mu\tilde{A}(T)}[l(X^{t,x}(T))-l(X^{t',x'}(T))]^2\right]+2E\left[\int_s^Te^{\mu\tilde{A}(r)}\bar{Y}(r)\bar{f}(r)dr\right]\\
&+2E\left[\int_s^Te^{\mu\tilde{A}(r)}\bar{Y}(r)\bar{h}(r)dA^{t',x'}(r)\right]+2E\left[\int_s^Te^{\mu\tilde{A}(r)}\bar{Y}(r)h(r, X^{t,x}(r), Y^{t,x}(r))d\bar{A}(r)\right]\\
 \leq &E\left[e^{\mu\tilde{A}(T)}[l(X^{t,x}(T))-l(X^{t',x'}(T))]^2\right]
 +CE\left[\int_0^Te^{\mu\tilde{A}(r)}\bar{Y}^2(r)dr\right]\\
&+\frac{1}{2}E\left[\int_s^Te^{\mu\tilde{A}(r)}\{\bar{Z}^2(r)+\|\bar{J}(r,\cdot)\|_\nu^2\}dr\right]+CE\left[\int_0^Te^{\mu\tilde{A}(r)}\bar{X}^2(r)dr\right]\\
&+CE\left[\int_0^Te^{\mu\tilde{A}(r)}h^2(r, X^{t,x}(r), Y^{t,x}(r))d\|\bar{A}\|(r)\right]+CE\left[\int_0^Te^{\mu\tilde{A}(r)}\bar{X}^2(r)dA^{t',x'}(r)\right]\\
&+\delta_1 E\left[\int_s^T e^{\mu\tilde{A}(r)}\bar{Y}^2(r)d\|\bar{A}\|(r)\right]+(\delta_2+2C_0) E\left[\int_s^T e^{\mu\tilde{A}(r)}\bar{Y}^2(r)dA^{t',x'}(r)\right].
\end{align*}
Choosing $\delta_1=\mu$ and $\delta_2=\mu-2C_0$,  and using  Gronwall's lemma and Burkholder-Davis-Gundy's inequality, we deduce that
\begin{align*}
E\left[\sup_{0\leq s\leq T}e^{\mu\tilde{A}(s)}\bar{Y}^2(s)\right]\leq & E\left[e^{\mu\tilde{A}(T)}[l(X^{t,x}(T))-l(X^{t',x'}(T))]^2\right]+CE\left[\int_0^Te^{\mu\tilde{A}(r)}\bar{X}^2(r)dr\right]\\
&+CE\left[\sup_{0\leq r\leq T}e^{\mu\tilde{A}(T)}(1+|X^{t,x}(r)|^2+|Y^{t,x}(r)|^2)\|\bar{A}\|(T)\right]\\
&+CE\left[\int_0^Te^{\mu\tilde{A}(r)}\bar{X}^2(r)dA^{t',x'}(r)\right].
\end{align*}
We define $Y^{t,x}(s)$ for all $s\in[0,T]$ by $Y^{t,x}(s)=Y^{t,x}(t)$ for $0\leq s\leq t$. Therefore, we have
\begin{align*}
e^{\mu\tilde{A}(0)}|Y^{t,x}(t)-Y^{t',x'}(t')|^2=e^{\mu\tilde{A}(0)}|Y^{t,x}(0)-Y^{t',x'}(0)|^2\leq E\left[\sup_{0\leq s\leq T}e^{\mu\tilde{A}(s)}\bar{Y}^2(s)\right].
\end{align*}
The result follows from Proposition \ref{pro3} and the continuity of the function $l$.
\end{proof}

From the argument of the proof of the comparison theorem for GBSDE's without jumps as in Pardoux and Zhang \cite{Pardoux1998}, we can obtain the following comparison theorem for GBSDE with jumps.
\begin{proposition}\label{comparison}
Let $l_1(X^{t,x}(T)), l_2(X^{t,x}(T))\in L^2(\mathcal{F}_T)$ and denote by $(Y_1,Z_1,J_1)$ and $(Y_2,Z_2,J_2)$ the unique solutions of equation (\ref{VGBSDE}) with final conditions $l_1(X^{t,x}(T))$ and $l_2(X^{t,x}(T))$, respectively.
Suppose $l_2(X^{t,x}(T))\geq l_1(X^{t,x}(T))$,
then under the assumptions (A4.1)-(A4.2) it follows $Y_2^{t,x}(s)\geq Y_1^{t,x}(s)$, for $t\leq s\leq T$.
\end{proposition}

Consider the following IPDE with nonlinear Neumann boundary condition:
\begin{align}\label{PDE}
\left\{\begin{array}{l}
\frac{\partial u}{\partial t}(t,x)+\mathcal{L}u(t,x)+f\left(t,x,u(t,x), (\nabla u\sigma)(t,x), \mathcal{A}u(t,x)\right)=0, \ \ \ (t,x)\in [0,T]\times \mho,\\
\frac{\partial u}{\partial \mathbf{n}}(t,x)+h(t,x,u(t,x))=0, \ \ \ (t,x)\in [0,T]\times \partial\mho,\\
u(T,x)=l(x), \ \ \ x\in \bar{\mho},
\end{array}
\right.
\end{align}
where $\frac{\partial u}{\partial \mathbf{n}}(t,x)=\nabla\phi\cdot\nabla u$ and the second-order integral-differential operator $\mathcal{L}$ (the infinitesimal generator of the Markov process $\{X^{t,x}(s), t\leq s\leq T\}$ is of the form
\begin{align*}
&\mathcal{L}u(s,x)=b(s,x)\cdot\nabla u(s,x)+\frac{1}{2}\mbox{tr}\{\sigma\sigma^*(s,x)\nabla^2u(s,x)\}\\
&+\int_{\mathcal{E}}\left[u(s,x+\gamma(s,x,e))-u(s,x)-\nabla u(s,x)\cdot\gamma(s,x,e)\right]\nu(de),
\end{align*}
and $\mathcal{A}$ is an integral operator defined as
\begin{align*}
\mathcal{A}u(s,x)=\int_{\mathcal{E}}\left[u(s,x+\gamma(s,x,e))-u(s,x)\right]\nu(de).
\end{align*}

Applying It\^{o}'s formula to $u(s, X^{t,x}(s))$, we can obtain
\begin{theorem}
Let $u\in C^{1,2}([0,T]\times \bar{\mho})$ be a classical solution of (\ref{PDE}) such that for some $C,q>0$, $|\nabla u(t,x)|\leq C(1+|x|^q)$. Then for any $(t,x)\in[0,T]\times\bar{\mho}$, $\left\{\left(u(s, X^{t,x}(s)), \right.\right.$\\
$\left.\left.(\nabla u\sigma)(s,X^{t,x}(s)), u(s, X^{t,x}(s-)+\gamma(s,X^{t,x}(s-),\cdot))-u(s, X^{t,x}(s-))\right), t\leq s\leq T\right\}$ is the solution of the GBSDE with jumps (\ref{VGBSDE}).
\end{theorem}

We now give the defintion of a viscosity solution of (\ref{PDE}).
\begin{definition}\label{DEF1}
(i) $u\in C([0,T]\times \bar{\mho})$ is called a viscosity subsolution of (\ref{PDE}) if  $u(T,x)\leq l(x)$, for all $x\in\bar{\mho}$ and moreover for any $\varphi\in C^{1,2}([0,T]\times \bar{\mho})$, whenever $(t,x)\in[0,T]\times\bar{\mho}$ is a global maximum point of  $u-\varphi$, then
\begin{align*}
\left\{\begin{array}{l}
\frac{\partial \varphi}{\partial t}(t,x)+\mathcal{L}\varphi(t,x)+f\left(t,x,u(t,x), (\nabla\varphi\sigma)(t,x), \mathcal{A}\varphi(t,x)\right)\geq0, \ \ \ \mbox{if}\ x\in  \mho,\\
\max\left\{\frac{\partial \varphi}{\partial t}(t,x)+\mathcal{L}\varphi(t,x)+f\left(t,x,u(t,x), (\nabla\varphi\sigma)(t,x), \mathcal{A}\varphi(t,x)\right), \right.\\
\ \ \ \ \ \  \ \left.\frac{\partial \varphi}{\partial \mathbf{n}}(t,x)+h(t,x,u(t,x))\right\}\geq0, \ \ \ \mbox{if}\ x\in  \partial\mho.
\end{array}
\right.
\end{align*}
(ii) $u\in C([0,T]\times \bar{\mho})$ is called a viscosity supersolution of (\ref{PDE}) if $u(T,x)\geq l(x)$, for all  $x\in\bar{\mho}$ and moreover for any $\varphi\in C^{1,2}([0,T]\times \bar{\mho})$, whenever $(t,x)\in[0,T]\times\bar{\mho}$ is a global minimum point of  $u-\varphi$, then
\begin{align*}
\left\{\begin{array}{l}
\frac{\partial \varphi}{\partial t}(t,x)+\mathcal{L}\varphi(t,x)+f\left(t,x,u(t,x), (\nabla\varphi\sigma)(t,x), \mathcal{A}\varphi(t,x)\right)\leq0, \ \ \ \mbox{if}\ x\in  \mho,\\
\min\left\{\frac{\partial \varphi}{\partial t}(t,x)+\mathcal{L}\varphi(t,x)+f\left(t,x,u(t,x), (\nabla\varphi\sigma)(t,x), \mathcal{A}\varphi(t,x)\right), \right.\\
\ \ \ \ \ \  \ \left.\frac{\partial \varphi}{\partial \mathbf{n}}(t,x)+h(t,x,u(t,x))\right\}\leq0, \ \ \ \mbox{if}\ x\in  \partial\mho.
\end{array}
\right.
\end{align*}
(iii) $u\in C([0,T]\times \bar{\mho})$ is called a viscosity solution of (\ref{PDE}) if it is both a viscosity subsolution and a supersolution
of (\ref{PDE}).
\end{definition}

We prove that the GBSDE with jumps (\ref{VGBSDE}) provides a viscosity solution of (\ref{PDE}).
\begin{theorem}\label{VGBSDET}
Define $u(t,x):=Y^{t,x}(t), (t,x)\in [0,T]\times\bar{\mho}$, then it is a viscosity solution of (\ref{PDE}).
\end{theorem}
\begin{proof}
It follows from Proposition \ref{continuous} that the function $u\in C([0,T]\times\bar{\mho})$. We shall only prove that $u$ is a viscosity subsolution. The proof of being a viscosity supersolution can be performed similarly. Apparently $u(T,x)=l(x)$. Let $\varphi\in C^{1,2}([0,T]\times\mathbb{R}^n)$, $(t,x)\in[0,T]\times\bar{\mho}$ such that $\varphi\geq u$ and $u(t,x)=\varphi(t,x)$.

We first consider the case $x\in\mho$.
Choose $0<\epsilon\leq T-t$ such that  $\{y: |y-x|\leq\epsilon\}\subset\mho$. Define
$$\tau=\inf\{s\geq t: |X^{t,x}(s)-x|\geq\epsilon\}\wedge (t+\epsilon).$$
For $t\leq s\leq t+\epsilon, e\in\mathcal{E}$, let
$$(Y^\tau(s), Z^\tau(s), J^\tau(s,e))=(Y^{t,x}(s\wedge \tau), \mathbf{1}_{[0,\tau]}(s)Z^{t,x}(s), \mathbf{1}_{[0,\tau]}(s)J^{t,x}(s,e)).$$
Then $(Y^\tau, Z^\tau, J^\tau)$ solves the following BSDE with jumps:
\begin{align*}
Y^\tau(s)=&u(\tau, X^{t,x}(\tau))+\int_s^{t+\epsilon}\mathbf{1}_{[0,\tau]}(r)f\left(r, X^{t,x}(r), Y^\tau(r), Z^\tau(r), \int_\mathcal{E}J^\tau(r,e)\nu(de)\right)dr\\
&-\int_s^{t+\epsilon}Z^\tau(r)dW(r)-\int_s^{t+\epsilon}\int_\mathcal{E}J^\tau(r,e)\tilde{N}(dr,de).
\end{align*}
Consider the following BSDE with jumps:
\begin{align*}
\bar{Y}^\tau(s)=&\varphi(\tau, X^{t,x}(\tau))+\int_s^{t+\epsilon}\mathbf{1}_{[0,\tau]}(r)f\left(r, X^{t,x}(r), \bar{Y}^\tau(r), \bar{Z}^\tau(r), \int_\mathcal{E}\bar{J}^\tau(r,e)\nu(de)\right)dr\\
&-\int_s^{t+\epsilon}\bar{Z}^\tau(r)dW(r)-\int_s^{t+\epsilon}\int_\mathcal{E}\bar{J}^\tau(r,e)\tilde{N}(dr,de).
\end{align*}
From $u\leq\varphi$, we derive with the help of the comparison theorem (see Proposition 2.6 in \cite{Barles1997}) that
\begin{align}\label{comp1}
\bar{Y}^\tau(t)\geq Y^\tau(t)=u(t,x)=\varphi(t,x).
\end{align}
On the other hand, by It\^{o}'s formula, we get
\begin{align*}
\varphi(s\wedge\tau, &X^{t,x}(s\wedge\tau))=\varphi(\tau, X^{t,x}(\tau))-\int_s^{t+\epsilon}\mathbf{1}_{[0,\tau]}(r)\left(\frac{\partial\varphi}{\partial r}+\mathcal{L}\varphi\right)(r, X^{t,x}(r))dr\\
&-\int_s^{t+\epsilon}\mathbf{1}_{[0,\tau]}(r)(\nabla\varphi\sigma)(r, X^{t,x}(r))dW(r)\\
&-\int_s^{t+\epsilon}\int_\mathcal{E}\mathbf{1}_{[0,\tau]}(r)[\varphi(r, X^{t,x}(r-)+\gamma(r, X^{t,x}(r-),e))- \varphi(r, X^{t,x}(r-))] \tilde{N}(dr,de).
\end{align*}
For $t\leq s\leq t+\epsilon$, $e\in\mathcal{E}$, we define
\begin{align*}
\hat{Y}^\tau(s)&=\bar{Y}^\tau(s)-\varphi(s\wedge\tau, X^{t,x}(s\wedge\tau)), \\
\hat{Z}^\tau(s)&=\bar{Z}^\tau(s)-\mathbf{1}_{[0,\tau]}(s)(\nabla\varphi\sigma)(s, X^{t,x}(s)), \\
\hat{J}^\tau(s,e)&=\bar{J}^\tau(s,e)-\mathbf{1}_{[0,\tau]}(s)[\varphi(s, X^{t,x}(s-)+\gamma(s, X^{t,x}(s-),e))- \varphi(s, X^{t,x}(s-))].
\end{align*}
Then, $(\hat{Y}^\tau, \hat{Z}^\tau, \hat{J}^\tau)$ solves the following BSDE with jumps:
\begin{align*}
\hat{Y}^\tau(s)=&\int_s^{t+\epsilon}\mathbf{1}_{[0,\tau]}(r)\left[\left(\frac{\partial\varphi}{\partial r}+\mathcal{L}\varphi\right)(r, X^{t,x}(r))+f\left(r, X^{t,x}(r), \bar{Y}^\tau(r), \bar{Z}^\tau(r), \int_\mathcal{E}\bar{J}^\tau(r,e)\nu(de)\right)\right]dr\\
&-\int_s^{t+\epsilon}\hat{Z}^\tau(r)dW(r)-\int_s^{t+\epsilon}\int_\mathcal{E}\hat{J}^\tau(r,e)\tilde{N}(dr,de).
\end{align*}
By using standard estimation method (see, e.g., \cite{Barles1997}), we obtain for some $C>0$,
\begin{align*}
E\left[\sup_{t\leq s\leq t+\epsilon}|\hat{Y}^\tau(s)|^2\right]&\leq C\epsilon,
\end{align*}
and
\begin{align*}
\frac{1}{\epsilon}\left(E\left[\int_t^{t+\epsilon}|Z^\tau(s)|^2ds\right]+E\left[\int_t^{t+\epsilon}\|J^\tau(s,\cdot)\|_\nu^2ds\right]\right)&\leq C\sqrt{\epsilon}.
\end{align*}
We now suppose that
\begin{align*}
\left(\frac{\partial \varphi}{\partial t}+\mathcal{L}\varphi\right)(t,x)+f\left(t,x,u(t,x), (\nabla\varphi\sigma)(t,x), \mathcal{A}\varphi(t,x)\right)<0,
\end{align*}
and we will find a contradiction.

In fact, there exist some $\epsilon_0>0$ and some $\delta>0$ such that for all $0<\epsilon\leq\epsilon_0$,
\begin{align*}
\Delta_\epsilon:=&\frac{1}{\epsilon}E\left[\int_t^{t+\epsilon}\mathbf{1}_{[0,\tau]}(r)\left\{\left(\frac{\partial\varphi}{\partial r}+\mathcal{L}\varphi\right)(r,X^{t,x}(r))\right.\right.\\
&\left.\left.+f\left(r,X^{t,x}(r),\varphi(r,X^{t,x}(r)), (\nabla \varphi\sigma)(r,X^{t,x}(r)), \mathcal{A}\varphi(r,X^{t,x}(r))\right)\right\}dr\right]\leq-\delta.
\end{align*}
From (\ref{comp1}), we have that $\hat{Y}^\tau(t)\geq0$. So that
\begin{align*}
\frac{1}{\epsilon}\hat{Y}^\tau(t)=&\frac{1}{\epsilon}E\left[\int_t^{t+\epsilon}\mathbf{1}_{[0,\tau]}(r)\left\{\left(\frac{\partial\varphi}{\partial r}+\mathcal{L}\varphi\right)(r, X^{t,x}(r))+f\left(r, X^{t,x}(r), \hat{Y}^\tau(r)+\varphi(r,X^{t,x}(r)), \right.\right.\right.\\
&\left.\left.\left.\hat{Z}^\tau(r)+(\nabla \varphi\sigma)(r,X^{t,x}(r)), \int_\mathcal{E}\hat{J}^\tau(r,e)\nu(de)+\mathcal{A}\varphi(r,X^{t,x}(r))\right)\right\}dr\right]\geq0.
\end{align*}
Hence, for all $0<\epsilon\leq\epsilon_0$, we get
\begin{align*}
0&<\delta\leq\Big|\frac{1}{\epsilon}\hat{Y}^\tau(t)-\Delta_\epsilon\Big|\\
&\leq C \left\{\left(E\left[\sup_{t\leq s\leq t+\epsilon}|\hat{Y}^\tau(s)|^2\right]\right)^{1/2}+\left(\frac{1}{\epsilon}E\left[\int_t^{t+\epsilon}|Z^\tau(s)|^2ds\right]\right)^{1/2}\right.\\
&\ \ \ \ \ \ \ \ \ \ \ \  \ \left.+\left(\frac{1}{\epsilon}E\left[\int_t^{t+\epsilon}\|J^\tau(s,\cdot)\|_\nu^2ds\right]\right)^{1/2}\right\}\\
&\leq C\epsilon^{1/4},
\end{align*}
which is a contradiction.

We now consider the case $x\in\partial\mho$.  Let  $0<\epsilon\leq T-t$.
For $t\leq s\leq t+\epsilon$, we have
\begin{align*}
Y^{t,x}(s)=&u(t+\epsilon, X^{t,x}(t+\epsilon))+\int_s^{t+\epsilon}f\left(r, X^{t,x}(r), Y^{t,x}(r), Z^{t,x}(r), \int_\mathcal{E}J^{t,x}(r,e)\nu(de)\right)dr\\
&-\int_s^{t+\epsilon}Z^{t,x}(r)dW(r)-\int_s^{t+\epsilon}\int_\mathcal{E}J^{t,x}(r,e)\tilde{N}(dr,de)\\
&+\int_s^{t+\epsilon}h(r, X^{t,x}(r), Y^{t,x}(r))dA^{t,x}(r).
\end{align*}
Consider the following BSDE with jumps
\begin{align*}
\bar{Y}^{t,x}(s)=&\varphi(t+\epsilon, X^{t,x}(t+\epsilon))+\int_s^{t+\epsilon}f\left(r, X^{t,x}(r), \bar{Y}^{t,x}(r), \bar{Z}^{t,x}(r), \int_\mathcal{E}\bar{J}^{t,x}(r,e)\nu(de)\right)dr\\
&-\int_s^{t+\epsilon}\bar{Z}^{t,x}(r)dW(r)-\int_s^{t+\epsilon}\int_\mathcal{E}\bar{J}^{t,x}(r,e)\tilde{N}(dr,de)\\
& +\int_s^{t+\epsilon}h(r, X^{t,x}(r), \bar{Y}^{t,x}(r))dA^{t,x}(r), \ \ \ t\leq s\leq t+\epsilon.
\end{align*}
Taking into account that $u\leq\varphi$, we derive from the Proposition \ref{comparison} that
\begin{align}\label{comp2}
\bar{Y}^{t,x}(t)\geq Y^{t,x}(t)=u(t,x)=\varphi(t,x).
\end{align}
On the other hand, from It\^{o}'s formula, we obtain
\begin{align*}
\varphi(s, &X^{t,x}(s))=\varphi(t+\epsilon, X^{t,x}(t+\epsilon))-\int_s^{t+\epsilon}\left(\frac{\partial\varphi}{\partial r}+\mathcal{L}\varphi\right)(r, X^{t,x}(r))dr\\
&-\int_s^{t+\epsilon}(\nabla\varphi\sigma)(r, X^{t,x}(r))dW(r)-\int_s^{t+\epsilon}\frac{\partial\varphi}{\partial \mathbf{n}}(r, X^{t,x}(r))dA^{t,x}(r)\\
&-\int_s^{t+\epsilon}\int_\mathcal{E}[\varphi(r, X^{t,x}(r-)+\gamma(r, X^{t,x}(r-),e))- \varphi(r, X^{t,x}(r-))] \tilde{N}(dr,de).
\end{align*}
For $t\leq s\leq t+\epsilon$, $e\in\mathcal{E}$, we define
\begin{align*}
\hat{Y}^{t,x}(s)&=\bar{Y}^{t,x}(s)-\varphi(s, X^{t,x}(s)), \\
\hat{Z}^{t,x}(s)&=\bar{Z}^{t,x}(s)-(\nabla\varphi\sigma)(s, X^{t,x}(s)), \\
\hat{J}^{t,x}(s,e)&=\bar{J}^{t,x}(s,e)-[\varphi(s, X^{t,x}(s-)+\gamma(s, X^{t,x}(s-),e))- \varphi(s, X^{t,x}(s-))].
\end{align*}
Then, $(\hat{Y}, \hat{Z}, \hat{J})$ solves the following GBSDE with jumps:
\begin{align*}
\hat{Y}^{t,x}(s)=&\int_s^{t+\epsilon}\left[\left(\frac{\partial\varphi}{\partial r}+\mathcal{L}\varphi\right)(r, X^{t,x}(r))+f\left(r, X^{t,x}(r), \bar{Y}^{t,x}(r), \bar{Z}^{t,x}(r), \int_\mathcal{E}\bar{J}^{t,x}(r,e)\nu(de)\right)\right]dr\\
&-\int_s^{t+\epsilon}\hat{Z}^{t,x}(r)dW(r)-\int_s^{t+\epsilon}\int_\mathcal{E}\hat{J}^{t,x}(r,e)\tilde{N}(dr,de)\\
&+\int_s^{t+\epsilon}\left[h(r, X^{t,x}(r), \bar{Y}^{t,x}(r))+\frac{\partial\varphi}{\partial \mathbf{n}}(r, X^{t,x}(r))\right]dA^{t,x}(r).
\end{align*}
By using standard estimation techniques, we have for some $C>0$,
\begin{align*}
E\left[\sup_{t\leq s\leq t+\epsilon}|\hat{Y}^{t,x}(s)|^2\right]&\leq C\epsilon,
\end{align*}
and
\begin{align*}
\frac{1}{\epsilon}\left(E\left[\int_t^{t+\epsilon}|\hat{Z}^{t,x}(s)|^2ds\right]+E\left[\int_t^{t+\epsilon}\|\hat{J}^{t,x}(s,\cdot)\|_\nu^2ds\right]\right)&\leq C\sqrt{\epsilon}.
\end{align*}
We now assume that
\begin{align*}
&\max\left\{\left(\frac{\partial \varphi}{\partial t}+\mathcal{L}\varphi\right)(t,x)+f\left(t,x,u(t,x), (\nabla\varphi\sigma)(t,x), \mathcal{A}\varphi(t,x)\right),\right.\\
& \ \ \ \ \ \ \ \ \left.h(t, x, u(t,x))+\frac{\partial\varphi}{\partial \mathbf{n}}(t, x)\right\}<0,
\end{align*}
and we will seek out a contradiction.

Indeed, there exist some $\epsilon_0>0$ and some $\delta>0$ such that for all $0<\epsilon\leq\epsilon_0$,
\begin{align*}
\Delta_\epsilon:=&\frac{1}{\epsilon}E\left[\int_t^{t+\epsilon}\left\{\left(\frac{\partial\varphi}{\partial r}+\mathcal{L}\varphi\right)(r,X^{t,x}(r))\right.\right.\\
&\left.+f\left(r,X^{t,x}(r),\varphi(r,X^{t,x}(r)), (\nabla \varphi\sigma)(r,X^{t,x}(r)), \mathcal{A}\varphi(r,X^{t,x}(r))\right)\right\}dr\\
&\left.+\int_t^{t+\epsilon}\left\{ h(r, X^{t,x}(r), \varphi(r,X^{t,x}(r)))+\frac{\partial\varphi}{\partial \mathbf{n}}(r, X^{t,x}(r))\right\}dA^{t,x}(r)\right]\leq-\delta.
\end{align*}
It follows from (\ref{comp2}) that $\hat{Y}^{t,x}(t)\geq0$. Hence
\begin{align*}
\frac{1}{\epsilon}\hat{Y}^{t,x}(t)=&\frac{1}{\epsilon}E\left[\int_t^{t+\epsilon}\left\{\left(\frac{\partial\varphi}{\partial r}+\mathcal{L}\varphi\right)(r, X^{t,x}(r))+f\left(r, X^{t,x}(r), \hat{Y}^{t,x}(r)+\varphi(r,X^{t,x}(r)), \right.\right.\right.\\
&\left.\left.\hat{Z}^{t,x}(r)+(\nabla \varphi\sigma)(r,X^{t,x}(r)), \int_\mathcal{E}\hat{J}^{t,x}(r,e)\nu(de)+\mathcal{A}\varphi(r,X^{t,x}(r))\right)\right\}dr\\
&\left.+\int_t^{t+\epsilon}\left\{ h(r, X^{t,x}(r), \hat{Y}^{t,x}(r)+\varphi(r,X^{t,x}(r)))+\frac{\partial\varphi}{\partial \mathbf{n}}(r, X^{t,x}(r))\right\}dA^{t,x}(r)\right]\geq0.
\end{align*}
Therefore, for all $0<\epsilon\leq\epsilon_0$, we have
\begin{align*}
0&<\delta\leq\Big|\frac{1}{\epsilon}\hat{Y}^{t,x}(t)-\Delta_\epsilon\Big|\\
&\leq C \left\{\left(E\left[\sup_{t\leq s\leq t+\epsilon}|\hat{Y}^{t,x}(s)|^2\right]\right)^{1/2}+\left(\frac{1}{\epsilon}E\left[\int_t^{t+\epsilon}|Z^\tau(s)|^2ds\right]\right)^{1/2}\right.\\
&\ \ \ \ \ \ \ \ \ \ \left.+\left(\frac{1}{\epsilon}E\left[\int_t^{t+\epsilon}\|J^{t,x}(s,\cdot)\|_\nu^2ds\right]\right)^{1/2}\right\}\\
&\leq C\epsilon^{1/4},
\end{align*}
which is impossible. This completes the proof.
\end{proof}

\section{Stochastic viscosity solutions to SIPDEs with nonlinear Neumann boundary conditions}\label{sec5}
In this section, we first introduce the notion of stochastic viscosity solutions to semi-linear SIPDEs with nonlinear Neumann boundary conditions. Then we give a probabilistic representation for the stochastic viscosity solutions. We consider the following SIPDE with nonlinear Neumann boundary condition:
\begin{align}\label{SPDE}
(f,g,h) \ \left\{\begin{array}{l}
du(t,x)+\left[\mathcal{L}u(t,x)+f\left(t,x,u(t,x), (\nabla u\sigma)(t,x), \mathcal{A}u(t,x)\right)\right]dt\\
 \  + g(t,x)\overleftarrow{d}B(t)=0, \ \ \ (t,x)\in [0,T]\times \mho,\\
\frac{\partial u}{\partial \mathbf{n}}(t,x)+h(t,x,u(t,x))=0, \ \ \ (t,x)\in [0,T]\times \partial\mho,\\
u(T,x)=l(x), \ \ \ x\in \bar{\mho},
\end{array}
\right.
\end{align}
where
\begin{align*}
&l: \bar{\mho}\to \mathbb{R},\ \ \
f: \Omega\times[0, T]\times \bar{\mho} \times \mathbb{R}^d \times \mathbb{R} \to \mathbb{R}, \\
&g: \Omega\times[0, T]\times \bar{\mho}  \to \mathbb{R}^m, \ \ \ h: \Omega\times[0, T]\times \bar{\mho}\to \mathbb{R}
\end{align*}
are continuous and satisfy the assumptions (A4.1)-(A4.2) for almost all $\omega\in\Omega$. Furthermore, we assume that\\
(A5.1) for some $C>0$, all $t\in[0,T]$ and $x, x'\in\bar{\mho}$:
\begin{align*}
& |g(t, x)|\leq C(1+|x|),\\
&|g(t, x)- g(t,x')|\leq C(|x- x'|).
\end{align*}
(A5.2) The function $g\in C_b^{0,2}([0,T]\times\bar{\mho}; \mathbb{R}^m)$.

For any $(t,x)\in [0,T]\times \bar{\mho}$, let $\{(Y^{t,x}(s),  Z^{t,x}(s),  J^{t,x}(s,\cdot)), t\leq s\leq T\}$ be the unique solution of the GBDSDE with jumps:
\begin{align}\label{VGBDSDE}
Y^{t,x}(s)=&l(X^{t,x}(T))+\int_s^Tf\left(r, X^{t,x}(r), Y^{t,x}(r), Z^{t,x}(r), \int_{\mathcal{E}}J^{t,x}(r,e)\nu(de)\right)dr\notag\\
&+\int_s^Tg(r, X^{t,x}(r))\overleftarrow{d}B(r)-\int_s^TZ^{t,x}(r)dW(r)-\int_s^T\int_{\mathcal{E}}J^{t,x}(r,e)\tilde{N}(dr,de)\notag\\
&+\int_s^Th(r,X^{t,x}(r), Y^{t,x}(r))dA^{t,x}(r), \ \ \ t\leq s\leq T,
\end{align}
where the process $X$ is the unique solution of \ref{RSDE}.

\begin{proposition}\label{RFC}
The random field $(t,x)\mapsto Y^{t,x}(t)$ is almost surely continuous in $(t,x)$.
\end{proposition}
\begin{proof}
The proof  is analogous to the proof of Proposition \ref{continuous} and we omit the details here.
\end{proof}

  Define
\begin{align*}
\mathcal{R}_{f}\varphi(s,x)=\mathcal{L}\varphi(s,x)+f\left(s,x,\varphi(s,x), (\nabla\varphi\sigma)(s,x), \mathcal{A}\varphi(s,x)\right).
\end{align*}

We now introduce the notion of a stochastic viscosity solution of (\ref{SPDE}).
\begin{definition}\label{DEF2}
(i) A random field $u\in \mathcal{C}(\mathbb{F}^B, [0,T]\times \bar{\mho})$ is called a stochastic viscosity subsolution of (\ref{SPDE}) if $u(T,x)\leq l(x)$, for all $x\in\bar{\mho}$ and moreover for any stopping time $\tau\in \mathcal{T}^B(0,T)$, any $\bar{\mho}$-valued state variable $\zeta \in \mathcal{F}_\tau^B$ and any random field $\varphi\in \mathcal{C}^{1,2}(\mathcal{F}_\tau^B, [0,T]\times \mho)$ such that for almost all $\omega\in\{0<\tau<T\}$,
\begin{align*}
u(\omega,t,x)-\eta(\omega,t,x,\varphi(t,x))\leq0=u(\tau(\omega),\zeta(\omega))-\eta(\tau(\omega),\zeta(\omega),\varphi(\tau(\omega),\zeta(\omega))),
\end{align*}
for all $(t,x)\in[0,T]\times\mho$, then
\begin{align*}
\left\{\begin{array}{l}
\mathcal{R}_{f}\psi(\tau,\zeta)+\frac{\partial \varphi}{\partial t}(\tau,\zeta)\geq0, \ \ \ \mbox{a.s. on}\ \{0<\tau<T\}\cap\{\zeta\in  \mho\},\\
\max\left\{\mathcal{R}_{f}\psi(\tau,\zeta)+\frac{\partial \varphi}{\partial t}(\tau,\zeta), \frac{\partial \psi}{\partial \mathbf{n}}(\tau,\zeta)+h(\tau, \zeta, \psi(\tau,\zeta))\right\}\geq0, \\
\ \ \ \ \ \ \ \ \ \mbox{a.s. on}\ \{0<\tau<T\}\cap\{\zeta\in  \partial\mho\},
\end{array}
\right.
\end{align*}
where $\psi(t,x):=\eta(t,x,\varphi(t,x))$ and the stochastic flow $\eta(t,x,y)$ is defined as the unique solution of the SDE with two spatial parameters:
$$\eta(t,x,y)=y+\int_t^Tg(s,x) \overleftarrow{d}B(s), \ \ \ 0\leq t\leq T.$$
(ii)  A random field $u\in \mathcal{C}(\mathbb{F}^B,[0,T]\times \bar{\mho})$ is called a stochastic viscosity supersolution of (\ref{SPDE}) if $u(T,x)\geq l(x)$, for all $x\in\bar{\mho}$ and moreover for any stopping time $\tau\in \mathcal{T}^B(0,T)$, any $\bar{\mho}$-valued state variable $\zeta \in \mathcal{F}_\tau^B$ and any random field $\varphi\in \mathcal{C}^{1,2}(\mathcal{F}_\tau^B,[0,T]\times \mho)$ such that for $\mathbb{P}$-almost all $\omega\in\{0<\tau<T\}$,
\begin{align*}
u(\omega,t,x)-\eta(\omega,t,x,\varphi(t,x))\geq0=u(\tau(\omega),\zeta(\omega))-\eta(\tau(\omega),\zeta(\omega),\varphi(\tau(\omega),\zeta(\omega))),
\end{align*}
for all $(t,x)\in[0,T]\times\mho$,  then
\begin{align*}
\left\{\begin{array}{l}
\mathcal{R}_{f}\psi(\tau,\zeta)+\frac{\partial \varphi}{\partial t}(\tau,\zeta)\leq0, \ \ \ \mbox{a.s. on}\ \{0<\tau<T\}\cap\{\zeta\in  \mho\},\\
\min\left\{\mathcal{R}_{f}\psi(\tau,\zeta)+\frac{\partial \varphi}{\partial t}(\tau,\zeta), \frac{\partial \psi}{\partial \mathbf{n}}(\tau,\zeta)+h(\tau, \zeta, \psi(\tau,\zeta))\right\}\leq0, \\
\ \ \ \ \ \ \ \ \ \mbox{a.s. on}\ \{0<\tau<T\}\cap\{\zeta\in  \partial\mho\}.
\end{array}
\right.
\end{align*}
(iii) A random field $u\in \mathcal{C}(\mathbb{F}^B,[0,T]\times \bar{\mho})$ is called a stochastic viscosity solution of (\ref{SPDE}) if it is both a stochastic viscosity subsolution and a supersolution
of (\ref{SPDE}).
\end{definition}

\begin{remark}
Under the condition $u(\tau,\zeta)=\eta(\tau,\zeta,\varphi(\tau,\zeta))$, we have
\begin{align*}
\mathcal{R}_{f}\psi(\tau,\zeta)=\mathcal{L}\psi(\tau,\zeta)+f\left(\tau,\zeta,u(\tau,\zeta), (\nabla \psi\sigma)(\tau,\zeta), \mathcal{A}\psi(\tau,\zeta)\right).
\end{align*}
Hence, if $f,h$ are deterministic and $g\equiv0$, the Definition \ref{DEF2} agrees with the Definition \ref{DEF1}.
\end{remark}

Next we prove a key proposition which allows us to provide a stochastic viscosity solution to the SIPDE (f,g,h) (\ref{SPDE}).
\begin{proposition}[Doss-Sussmann transformation]\label{keypro}
A random field $u$ is a stochastic viscosity solution to the SIPDE $(f,g,h)$ (\ref{SPDE}) if and only if $v(t,x)=\varepsilon(t,x,u(t,x))$ is a stochastic viscosity solution to the SIPDE $(\tilde{f},0,\tilde{h})$, where
\begin{align}
\mathcal{\varepsilon}(t,x,y)&=y-\int_t^Tg(s,x) \overleftarrow{d}B(s),\notag\\
\tilde{f}(t,x,y,z,j)&= f\left(t,x,\eta(t,x,y), (\nabla_x\eta\sigma)(t,x,y)+z, \mathcal{A}_x\eta(t,x,y)+j\right)+\mathcal{L}_x\eta(t,x,y),\label{f}\\
\tilde{h}(t,x,y)&=h(t,x,\eta(t,x,y))+\frac{\partial \eta}{\partial \mathbf{n}}(t,x,y),\label{h}
\end{align}
where
\begin{align*}
&\mathcal{L}_x\eta(t,x,y)=b(t,x)\cdot\nabla_x \eta(t,x,y)+\frac{1}{2}\mbox{tr}(\sigma\sigma^*(t,x)\nabla_x^2\eta(t,x,y))\\
&+\int_{\mathcal{E}}\left[\eta(t,x+\gamma(t,x,e),y)-\eta(t,x,y)-\nabla_x \eta(t,x,y)\cdot\gamma(t,x,e)\right]\nu(de),
\end{align*}
and
\begin{align*}
\mathcal{A}_x\eta(t,x,y)=\int_{\mathcal{E}}\left[\eta(t,x+\gamma(t,x,e),y)-\eta(t,x,y)\right]\nu(de).
\end{align*}
\end{proposition}
\begin{proof}
We only give the proof for the stochastic subsolution case. The supersolution part can be proved similarly. We assume that
the random field $u\in \mathcal{C}(\mathbb{F}^B, [0,T]\times \bar{\mho})$ is a stochastic viscosity subsolution of the SIPDE $(f,g,h)$, then
$v\in \mathcal{C}([0,T]\times \bar{\mho}; \mathbb{F}^B)$. Let $\tau\in \mathcal{T}^B(0,T)$, $\zeta\in\mathcal{F}_\tau^B$ and $\varphi\in \mathcal{C}^{1,2}(\mathcal{F}_\tau^B,[0,T]\times\mathbb{R})$ be such that for almost all $\omega\in\{0<\tau<T\}$,
\begin{align*}
v(\omega, t,x)-\varphi(\omega, t,x)\leq 0=v(\tau(\omega), \zeta(\omega))-\varphi(\tau(\omega), \zeta(\omega)),
\end{align*}
for all $(t,x)\in[0,T]\times\mho$.

We note that under the assumption (A5.2) the random field $\eta\in\mathcal{C}^{0,2}(\mathbb{F}^B,[0,T]\times\bar{\mho})$ as well as $\mathcal{\varepsilon}\in\mathcal{C}^{0,2}(\mathbb{F}^B,[0,T]\times\bar{\mho})$. Let $\psi(t,x)=\eta(t,x,\varphi(t,x))$. Since $y\mapsto\eta(t,x,y)$ is strictly increasing, we have
\begin{align*}
u(t,x)-\psi(t,x)&=\eta(t,x,v(t,x))-\eta(t,x,\varphi(t,x))\\
&\leq 0=\eta(\tau,\zeta,v(\tau,\zeta))-\eta(\tau,\zeta,\varphi(\tau,\zeta))=u(\tau,\zeta)-\psi(\tau,\zeta)
\end{align*}
for all $(t,x)\in[0,T]\times\mho$, a.s. on $\{0<\tau<T\}$. Furthermore, since $u$ is a stochastic viscosity subsolution of the SIPDE $(f,g,h)$, we get
\begin{align*}
\mathcal{R}_f\psi(\tau,\zeta)+\frac{\partial\varphi}{\partial t}(\tau,\zeta)\geq0,  \ \ \ \mbox{a.s. on}\  \{0<\tau<T\}.
\end{align*}
Since
\begin{align*}
\mathcal{L}\psi(t,x)
=&b(t,x)\cdot\nabla\psi(t,x)+\frac{1}{2}\mbox{tr}\{\sigma\sigma^*(t,x)\nabla^2\psi(t,x)\}\\
&+\int_{\mathcal{E}}\left\{\psi(t,x+\gamma(t,x,e))-\psi(t,x)-\nabla\psi(t,x)\cdot\gamma(t,x,e)\right\}\nu(de)\\
=&\mathcal{L}_x\eta(t,x,\varphi(t,x))+\mathcal{L}\varphi(t,x)
\end{align*}
and
\begin{align*}
f&\left(t,x,\psi(t,x), (\nabla\psi\sigma)(t,x), \mathcal{A}\psi(t,x)\right)\\
&=f\left(t,x,\eta(t,x, \varphi(t,x)), \left[\nabla_x \eta(t,x, \varphi(t,x))+\nabla\varphi(t,x)\right]\sigma(t,x), \mathcal{A}_x\eta(t,x,\varphi(t,x))+\mathcal{A}\varphi(t,x)\right),
\end{align*}
we obtain
\begin{align*}
\mathcal{R}_f\psi(t,x)=\mathcal{R}_{\tilde{f}}\varphi(t,x).
\end{align*}
So that
\begin{align*}
\mathcal{R}_{\tilde{f}}\varphi(t,x)+\frac{\partial\varphi}{\partial t}(t,x)\geq0, \ \ \ \mbox{a.s. on} \ \{0<\tau<T\}.
\end{align*}
In addition, we notice that for all $(t,x)\in[0,T]\times\partial\mho$,
\begin{align*}
\frac{\partial\psi}{\partial \mathbf{n}}(\tau,\zeta)+h(\tau,\zeta,\psi(\tau,\zeta))=\frac{\partial\varphi}{\partial \mathbf{n}}(\tau,\zeta)+\tilde{h}(\tau,\zeta,\varphi(\tau,\zeta)).
\end{align*}
Hence, we get
\begin{align*}
\max\left\{\mathcal{R}_{\tilde{f}}\varphi(t,x)+\frac{\partial \varphi}{\partial t}(\tau,\zeta), \frac{\partial \varphi}{\partial \mathbf{n}}(\tau,\zeta)+\tilde{h}(\tau, \zeta, \varphi(\tau,\zeta))\right\}\geq0,
\end{align*}
a.s. on $\{0<\tau<T\}\cap\{\zeta\in  \partial\mho\}$. Thus, the random field $v$ is a stochastic viscosity subsolution of the SIPDE $(\tilde{f},0,\tilde{h})$.
\end{proof}

We are in a position to prove the existence of the stochastic viscosity solutions to  semilinear SIPDEs with nonlinear Neumann boundary conditions.
For $t\in[0,T]$ and $x\in\bar{\mho}$, let
\begin{align}
U^{t,x}(s)=&\varepsilon(s,X^{t,x}(s), Y^{t,x}(s)),  \ \ \ 0\leq t\leq s\leq T,\label{U}\\
V^{t,x}(s)=&(\nabla_x \varepsilon\sigma)(s, X^{t,x}(s), Y^{t,x}(s))+Z^{t,x}(s), \ \ \ 0\leq t\leq s\leq T, \label{V}\\
M^{t,x}(s,\cdot)=&\varepsilon(s, X^{t,x}(s-)+\gamma(s, X^{t,x}(s-), \cdot), Y^{t,x}(s))-\varepsilon(s, X^{t,x}(s-), Y^{t,x}(s))\notag\\
&  +J^{t,x}(s,\cdot), \ \ \ 0\leq t\leq s\leq T. \label{M}
\end{align}
Then we have the following result:

\begin{theorem}
For each $(t,x)\in[0,T]\times\bar{\mho}$, the process $\{(U^{t,x}(s), V^{t,x}(s), M^{t,x}(s,\cdot)), t\leq s\leq T\}\in V^2(\mathbb{F},[t,T];\mathbb{R}\times\mathbb{R}^d\times L^2(\mathcal{E},\nu))$ is the unique solution to the following GBSDE with jumps:
\begin{align}\label{GBSDEJ}
U^{t,x}(s)=&l(X^{t,x}(T))+\int_s^T\tilde{f}\left(r, X^{t,x}(r), U^{t,x}(r), V^{t,x}(r), \int_{\mathcal{E}}M^{t,x}(r,e)\nu(de)\right)dr\notag\\
&-\int_s^TV^{t,x}(r)dW(r)-\int_s^T\int_{\mathcal{E}}M^{t,x}(r,e)\tilde{N}(dr,de)\notag\\
&+\int_s^T\tilde{h}(r,X^{t,x}(r), U^{t,x}(r))dA^{t,x}(r), \ \ \ t\leq s\leq T,
\end{align}
where $\tilde{f}$ and $\tilde{h}$ are respectively given by (\ref{f}) and (\ref{h}).
\end{theorem}
\begin{proof}
From Proposition 3.4 in Buckdahn and Ma \cite{Buckdahn2001}, we get
\begin{align*}
(U, V,M)\in V^2(\mathbb{F},[t,T];\mathbb{R}\times\mathbb{R}^d\times L^2(\mathcal{E},\nu)).
\end{align*}
Applying the generalized It\^{o}-Ventzell formula gives
\begin{align*}
U(s)=&l(X^{t,x}(T))-\int_s^Tg(s, X^{t,x}(s))\overleftarrow{d}B(s)-\int_s^T\nabla_x\varepsilon(r,X^{t,x}(r),Y^{t,x}(r))\cdot b(r,X^{t,x}(r))dr\\
&-\int_s^T(\nabla_x\varepsilon\sigma)(r,X^{t,x}(r),Y^{t,x}(r))dW(r)-\int_s^T\frac{\partial\varepsilon}{\partial \mathbf{n}}(r,X^{t,x}(r),Y^{t,x}(r))dA^{t,x}(r)\\
&-\int_s^T\int_{\mathcal{E}}\nabla_x\varepsilon(r,X^{t,x}(r),Y^{t,x}(r))\cdot\gamma(r,X^{t,x}(r),e)\tilde{N}(dr,de)\\
&-\frac{1}{2}\int_s^T\mbox{tr}\{(\sigma\sigma^*)(r,X^{t,x}(r))\nabla_x^2\varepsilon(r,X^{t,x}(r),Y^{t,x}(r))\}dr\\
&-\int_s^T\int_{\mathcal{E}}\{\varepsilon(r, X^{t,x}(r-)+\gamma(r,X^{t,x}(r-),e),Y^{t,x}(r))\\
&-\varepsilon(r, X^{t,x}(r-),Y^{t,x}(r))-\nabla_x\varepsilon(r, X^{t,x}(r-),Y^{t,x}(r))\cdot\gamma(r,X^{t,x}(r-),e)\}N(dr,de)\\
&+\int_s^Tf\left(r, X^{t,x}(r), Y^{t,x}(r), Z^{t,x}(r), \int_{\mathcal{E}}J^{t,x}(r,e)\nu(de)\right)dr\notag\\
&+\int_s^Tg(r, X^{t,x}(r))\overleftarrow{d}B(r)-\int_s^TZ^{t,x}(r)dW(r)-\int_s^T\int_{\mathcal{E}}J^{t,x}(r,e)\tilde{N}(dr,de)\notag\\
&+\int_s^Th(r,X^{t,x}(r), Y^{t,x}(r))dA^{t,x}(r)\\
=&l(X^{t,x}(T))+\int_s^T\mathfrak{F}\left(r, X^{t,x}(r), Y^{t,x}(r), Z^{t,x}(r), \int_{\mathcal{E}}J^{t,x}(r,e)\nu(de)\right)\\
&+\int_s^T\mathfrak{H}(r, X^{t,x}(r), Y^{t,x}(r))dA^{t,x}(r)-\int_s^TV^{t,x}(r)dW(r)\\
&-\int_s^T\int_{\mathcal{E}} M^{t,x}(r,e)\tilde{N}(dr,de),
\end{align*}
where
\begin{align*}
\mathfrak{F}(s,x,y,z,j):=&-\nabla_x\varepsilon(s,x,y)\cdot b(s,x)-\frac{1}{2}\mbox{tr}\{(\sigma\sigma^*)(s,x)\nabla_x^2\varepsilon(s,x,y)\}+f(s,x,y,z,j)\\
&-\int_{\mathcal{E}}\left\{\varepsilon(s, x+\gamma(r,x,e),y)-\varepsilon(s, x,y)-\nabla_x\varepsilon(s, x,y)\cdot\gamma(r,x,e)\right\}\nu(de),\\
\mathfrak{H}(s,x,y):=&h(s,x,y)-\frac{\partial\varepsilon}{\partial \mathbf{n}}(s,x,y).
\end{align*}
By (\ref{U})-(\ref{M}), we get
\begin{align}
Y^{t,x}(s)&=\eta(s, X^{t,x}(s), U^{t,x}(s)),\label{Y}\\
Z^{t,x}(s)&=V^{t,x}(s)+(\nabla_x\eta\sigma)(s,X^{t,x}(s), U^{t,x}(s)),\label{Z}\\
J^{t,x}(s,\cdot)&=M^{t,x}(s,\cdot)+\eta(s,X^{t,x}(s-)+\gamma(s, X^{t,x}(s-),\cdot),U^{t,x}(s))-\eta(s,X^{t,x}(s-),U^{t,x}(s)).\label{J}
\end{align}
From the definitions of $\varepsilon$ and $\eta$, we derive
\begin{align}
\nabla_x\varepsilon(s,X^{t,x}(s), Y^{t,x}(s))&=-\nabla_x\eta(s,X^{t,x}(s), U^{t,x}(s)),\label{DX}\\
\nabla_x^2\varepsilon(s,X^{t,x}(s), Y^{t,x}(s))&=-\nabla_x^2\eta(s,X^{t,x}(s), U^{t,x}(s)). \label{DXX}
\end{align}
Hence, form (\ref{Y})-(\ref{DXX}), we obtain
\begin{align*}
\mathfrak{F}&\left(s, X^{t,x}(s), Y^{t,x}(s), Z^{t,x}(s), \int_{\mathcal{E}}J^{t,x}(s,e)\nu(de)\right)\\
=&\nabla_x\eta(s,X^{t,x}(s), U^{t,x}(s))\cdot b(s,X^{t,x}(s))+\frac{1}{2}\mbox{tr}\{(\sigma\sigma^*)(s,X^{t,x}(s))\nabla_x^2\eta(s,X^{t,x}(s), U^{t,x}(s))\}\\
&+f\left(s,X^{t,x}(s),\eta(s, X^{t,x}(s), U^{t,x}(s)), V^{t,x}(s)+(\nabla_x\eta\sigma)(s,X^{t,x}(s), U^{t,x}(s)),\right.\\
&\left.\int_{\mathcal{E}}\{M^{t,x}(s,e)+\eta(s, X^{t,x}(s-)+\gamma(s,X^{t,x}(s-),e), U^{t,x}(s))-\eta(s, X^{t,x}(s-), U^{t,x}(s))\}\nu(de)\right)\\
&+\int_{\mathcal{E}}\{\eta(s,X^{t,x}(s-)+\gamma(s,X^{t,x}(s-),e), U^{t,x}(s))-\eta(s,X^{t,x}(s-), U^{t,x}(s))\}\nu(de)\\
=&\tilde{f}\left(s, X^{t,x}(s),U^{t,x}(s)), V^{t,x}(s), \int_{\mathcal{E}}M^{t,x}(s,e)\nu(de)\right)
\end{align*}
and
\begin{align*}
\mathfrak{H}(s,X^{t,x}(s), Y^{t,x}(s)):=&h(s,X^{t,x}(s), Y^{t,x}(s))-\frac{\partial\varepsilon}{\partial \mathbf{n}}(s,X^{t,x}(s), Y^{t,x}(s))\\
=&h(s,X^{t,x}(s), \eta(s,X^{t,x}(s), U^{t,x}(s)))+\frac{\partial\eta}{\partial \mathbf{n}}(s,X^{t,x}(s), U^{t,x}(s))\\
=&\tilde{h}(s,X^{t,x}(s), U^{t,x}(s)).
\end{align*}
This completes the proof.
\end{proof}

We are now ready to give the main result of this article.
\begin{theorem}
Define $u(t,x):=Y^{t,x}(t), (t,x)\in[0,T]\times\bar{\mho}$. Then the random field $u$ is a stochastic viscosity solution to the SIPDE $(f,g,h)$ (\ref{SPDE}).
\end{theorem}
\begin{proof}
From Proposition \ref{RFC}, we see that the random field $(t,x)\mapsto Y^{t,x}(t)$ is continuous. Since $Y^{t,x}(s)$ is $\mathcal{F}_{t,s}^W\vee \mathcal{F}_{t,s}^N\vee \mathcal{F}_{s,T}^B$-measurable, it follows that $Y^{t,x}(t)$ is $\mathcal{F}_{t,T}^B$-measurable. Thus $u(t,x)\in \mathcal{C}(\mathbb{F}^B,[0,T]\times\bar{\mho})$.  For a fixed $\omega\in\Omega$, since $(U^{t,x}(\omega,\cdot),V^{t,x}(\omega,\cdot),M^{t,x}(\omega,\cdot))$ is the unique solution of (\ref{GBSDEJ}) with coefficients $(\tilde{f}(\omega,\cdot,\cdot,\cdot,\cdot,\cdot),\tilde{h}(\omega,\cdot,\cdot,\cdot))$, it follows from Theorem \ref{VGBSDET} that
$U^{t,x}(\omega,t)=\varepsilon(\omega, t, X^{t,x}(\omega, t), Y^{ t,x}(\omega,t))$ is a viscosity solution to IPDE with coefficients $(\tilde{f}(\omega,\cdot,\cdot,\cdot,\cdot,\cdot),\tilde{h}(\omega,\cdot,\cdot,\cdot))$. Hence, by Definition 2.2. in Buckdahn and Ma \cite{Buckdahn2001}, $\varepsilon(\omega, t, x, u(\omega,t,x))$ is an $\omega$-wise viscosity solution to the SIPDE $(\tilde{f},0,\tilde{h})$. Therefore, from Proposition \ref{keypro}, we obtain that $u(t,x)$ is a stochastic viscosity solution to the SIPDE $(f,g,h)$ (\ref{SPDE}).
\end{proof}

\section{Singular stochastic optimal control}\label{sec6}
In this section, we study the singular stochastic optimal control
problem. We assume that the dynamics of a stochastic system is
modelled by the controlled GBDSDE with jumps  of the form
\begin{align}\label{SCGBDSDE}
\left\{\begin{array}{lcl} d X(t) & = & -f(t, X(t), Y(t), Z(t,\cdot))dt-Y(t)dW(t) +g(t, X(t))\overleftarrow{d}B(t)\\
\ & \ &-
\int_{\mathcal{E}}Z(t,e)\tilde{N}(dt,de)+h(t)dA(t), \ \ t\in[0,T],\\
X(T) &=&\xi,
\end{array}
\right.
\end{align}
where
$$f: \Omega\times[0,T]\times \mathbb{R} \times \mathbb{R}^d\times L^2(\mathcal{E},\nu)\to
 \mathbb{R},$$
$$g: \Omega\times \mathbb{R}\to \mathbb{R}^m, $$
$$h: \Omega\times[0,T]\to
 \mathbb{R},$$
 are  jointly measurable processes.
Here, we regard the process $A$ as the control process.

 The performance functional is given by
\begin{align*}
J(A)=E\left[\int_0^TF(t, X(t))dt+G(X(0))+\int_0^TH(t)dA(t)\right],
\end{align*}
where
$$F: \Omega\times[0,T]\times \mathbb{R} \to
 \mathbb{R},$$
$$G: \Omega\times \mathbb{R}\to \mathbb{R}, $$
$$H: \Omega\times[0,T]\to
 \mathbb{R},$$
 are given three jointly measurable processes.

A control process $A=\{A(t), t\in[0,T]\}$ is said to be
admissible if there exists a unique solution of (\ref{SCGBDSDE}) and
$$E\left[\int_0^T|F(t, X(t))|dt+|G(X(0))|+\int_0^T|H(t)|dA(t)\right]<\infty.$$
We denote the set of admissible controls by
$\mathcal{A}$. The optimal control problem is to find $A^*\in \mathcal{A}$ such
that
\begin{equation}\label{problem}
J(A^*)=\inf_{A\in\mathcal{A}}J(A).
\end{equation}

Suppose there exists a unique solution $(X,Y,Z)\in V^2(\mathbb{F},[0,T];\mathbb{R}\times\mathbb{R}^d\times L^2(\mathcal{E},\nu))$ of (\ref{SCGBDSDE}). For conditions which are sufficient to get existence and uniqueness of a solution of the equation (\ref{SCGBDSDE}), see
 the Section \ref{sec3}.

We define the Hamiltonian $\mathcal{H}: \Omega\times [0,T]\times
\mathbb{R}\times \mathbb{R}^d\times L^2(\mathcal{E},\nu)\times
\mathbb{R}\times \mathbb{R}^m\rightarrow \mathbb{R}$ as follows
\begin{align}\label{Hami}
\mathcal{H}(t,x,y,z, p,q)(dt,dA(t))=[F(t,x)-f(t, x,y,z)p-g(t,x)\cdot q]dt+[H(t)-ph(t)]dA(t),
\end{align}

We assume that $f$ and $g$ are continuously differentiable with respect to $x,y,z$, and $F$ and $G$ are continuously differentiable with respect to $x$. Therefore,  the Hamiltonian $\mathcal{H}$ is
continuously differentiable with respect  to $x,y,z$.

To simplify our notation, in what follows, we denote
\begin{align*}
&f(t)= f(t, X(t), Y(t), Z(t,\cdot)),\ \
f(t,e)= f(t, X(t), Y(t), Z(t,e)),\\
&f_x(t)=\frac{\partial f}{\partial x}(t, X(t), Y(t), Z(t,\cdot))
\end{align*}
and similarly to $f_y$, $f_z$, $g$, $g_x$, $g_y$, $F$, $F_x$, $G$, $G_x$, $\mathcal{H}$, $\mathcal{H}_x$, $\mathcal{H}_y$ and $\mathcal{H}_z$.

To derive the maximum principles, we
introduce the adjoint equation associated to the controlled system
(\ref{SCGBDSDE})
 governing the unknown $\mathcal{F}_t$-measurable
processes $(p(t), q(t))$ as the following forward doubly stochastic differential equation (FDSDE):
\begin{align}\label{ad}
\left\{\begin{array}{lcl} d p(t) & = &
-\mathcal{H}_x(t)(dt, dA(t))-q(t)\overleftarrow{d}B(t)-\mathcal{H}_y(t)dW(t)-\int_{\mathcal{E}}\mathcal{H}_z(t,e)\tilde{N}(dt,de),\\
p(0) &=&-G_x(0).
\end{array}
\right.
\end{align}


We suppose there exists a unique solution $(p,q)\in L^2(\mathbb{F},[0,T])\times L^2(\mathbb{F},[0,T]; \mathbb{R}^m)$ of the adjoint equation (\ref{ad}) corresponding to each admissible triple $(X,Y,Z)$. Readers may  apply the technology of time reversal on jump diffusion processes (see Jacod and Protter \cite{Jacod1988} or Wu and Liu \cite{Wu2018}) to derive the conditions under which the adjoint equation (\ref{ad}) has a unique solution.

\subsection{Sufficient Maximum Principle}
In this subsection, we establish a sufficient maximum principle for the
 backward doubly stochastic optimal control problem
(\ref{problem}).
\begin{theorem}\label{SMP}{(Sufficient stochastic maximum principle)}\\
Let $A^* \in{\mathcal A}$ with the corresponding solution
$(X^*(t),Y^*(t),Z^*(t,\cdot))$ of the system equation
\eqref{SCGBDSDE}, and $(p^*(t),q^*(t))$ of the adjoint
equation \eqref{ad}.
Suppose that\\
(i) The Hamiltonian $\mathcal{H}$ is convex in $x, y, z$  and
the cost function $G$ is convex in $x$.\\
(ii)
\begin{align}\label{smpcon}
\mathcal{H}(t,X^*(t),p^*(t),q^*(t))(dt,dA^*(t))=\inf_{A\in \mathcal{A}}
\mathcal{H}(t,X^*(t),p^*(t),q^*(t))(dt,dA(t)),
\end{align}
a.s., for almost all $t\in [0,T]$,  i.e.,
\begin{align*}
[H(t)-p^*(t)h(t)]dA(t)\geq [H(t)-p^*(t)h(t)]dA^*(t),
\end{align*}
for all $A\in \mathcal{A}$.

Then $A^*$ is an optimal
control.
\end{theorem}
\begin{proof}
By considering a suitable
increasing sequence of stopping times converging to $T$ as arguing
in \cite{Dahl2016}, we see that all the local martingales
appearing in the proof below are martingales.
To simplify our notation,  for any control process $A\in{\mathcal{A}}$ with corresponding
$(X(t), Y(t), Z(t,\cdot))$, we denote
$$\mathcal{H}(t)=\mathcal{H}(t, X(t), Y(t), Z(t,\cdot), p^*(t), q^*(t)), \ \ \mathcal{H}^*(t)=\mathcal{H}(t, X^*(t), Y^*(t), Z^*(t,\cdot), p^*(t), q^*(t)),$$
$$F^*(t)=F(t,X^*(t)), \ \ G^*(0)=G(X^*(0)),  \ \ f^*(t)=f(t, X^*(t),Y^*(t),Z^*(t,\cdot)), \ \
g^*(t)=g(t,X^*(t)).$$

 Consider
$$J(A)-J(A^*)=I_1+I_2+I_3,$$ where
\begin{eqnarray*}
I_1=E\left[\int_0^T(F(t)-F^*(t))dt\right],
\end{eqnarray*}
$$
I_2=E[G(0)-G^*(0)],
$$
$$I_3=E\left[\int_0^TH(t)(dA(t)-dA^*(t))\right].$$
Since $G$ is convex with respect to $x$, we have
\begin{align}\label{6.2}
I_2\geq
E[G_x(X^*(0))(X(0)-X^*(0))]=E[-p^*(0)(X(0)-X^*(0))].
\end{align}
Applying It\^o's formula to $p^*(t)(X(t)-X^*(t))$ yields
\begin{align}\label{6.3}
E[&-p^*(0)(X(0)-X^*(0))]=-E\left[\int_0^Tp^*(t-)(h(t)dA(t)-h(t)dA^*(t))\right]\notag\\
&-E\left[\int_0^T(X(t)-X^*(t))\mathcal{H}_x^*(t)dt\right]-E\left[\int_0^T(Y(t)-Y^*(t))\mathcal{H}_y^*(t)dt\right]\notag\\
&-E\left[\int_0^T\int_{\mathcal{E}}(Z(t,e)-Z^*(t,e))\mathcal{H}_z^*(t,e)\nu(de)dt\right].
\end{align}
By the definition of the Hamiltonian $\mathcal{H}$, we get
\begin{align}\label{6.4}
I_1= &E\left[\int_0^T\left\{\mathcal{H}(t-)(dt,dA(t))-\mathcal{H}^*(t-)(dt,dA^*(t))+\int_0^Tp^*(t)(f(t)-f^*(t))dt\right.\right.\notag\\
& \ \ \left.\left.+\int_0^Tp^*(t)(g(t)-g^*(t))dt -(H(t)-h(t)p^*(t-))(dA(t)-dA^*(t))\right\}\right].
\end{align}
Combining (\ref{6.2})-(\ref{6.4}) gives
\begin{align*}
J(A)-J(A^*)\geq &E\left[\int_0^T\left\{\mathcal{H}(t-)(dt,dA(t))
-\mathcal{H}^*(t-)(dt,dA^*(t))\right.\right.\notag\\
& -(X(t-)-X^*(t-))\mathcal{H}_x^*(t-)(dt,dA^*(t))\\
&-(Y(t-)-Y^*(t-))\mathcal{H}_y^*(t-)(dt,dA^*(t))\\
& \left.\left. -\int_{\mathcal{E}} (Z(t-,e)-Z^*(t-,e))\mathcal{H}_z^*(t-,e)(dt,dA^*(t))\nu(de)\right\}\right].
\end{align*}
By (\ref{smpcon}) and using the convexity of the Hamiltonian $\mathcal{H}$ lead to
\begin{align*}
J(A)-J(A^*)\geq  0.
\end{align*}
Since this holds for all $A\in{\mathcal{A}}$, this proves that $A^*$
is optimal.
\end{proof}

\subsection{Necessary Maximum Principle}
In this subsection, we give necessary conditions for the optimal
control problem.
We make the following assumptions:\\
(A6.1) For all $t\in[0,T]$ and all bounded $\mathcal{F}_t$-measurable $\alpha(\omega)$, the control $A(s):=\alpha\delta_t(s)$, $s\in[0,T]$ belongs to $\mathcal{A}$, where $\delta_t(\cdot)$ is the unit point mass at $t$.\\
(A6.2) For all $A, A'\in \mathcal{A}$ there exists $\delta>0$ such that the control $A+\rho A'\in \mathcal{A}$ for all $\rho\in(-\delta,\delta)$.

 Suppose that $A^*$ is the optimal control with
corresponding state processes $(X^*(t),Y^*(t),Z^*(t,\cdot))$ and the corresponding solution $(p^*(t), q^*(t))$ to the adjoint equation. For any
$A\in\mathcal A$, define $A^{\rho}(\cdot):=A^*(\cdot)+\rho A(\cdot)$, with the corresponding state
process $(X^{\rho}(t),Y^{\rho}(t),Z^{\rho}(t,\cdot))$.
To simplify our notation, we denote $f^*(t)=f(t, X^*(t),Y^*(t),Z^*(t,\cdot))$, $f^\rho(t)=f(t, X^\rho(t),Y^\rho(t),Z^\rho(t,\cdot))$
and similarly for $g$, $F$ and $G$.

\begin{lemma}
\begin{align}
&\lim_{\rho\to 0}
E\left[\sup_{t\in[0,T]}|X^{\rho}(t)-X^*(t)|^2\right]=0,\label{7.1}\\&\lim_{\rho\to
0}
E\left[\int_0^T|Y^{\rho}(t)-Y^*(t)|^2dt\right]=0,\label{7.2}\\&\lim_{\rho\to
0}
E\left[\int_0^T\|Z^{\rho}(t,\cdot)-Z^*(t,\cdot)\|_{\nu}^2dt\right]=0.\label{7.3}
\end{align}
\end{lemma}
\begin{proof}
Applying It\^ o's formula to $|X^{\rho}(t)-X^*(t)|^2$  and taking the expectation yield
\begin{align*}
&E\left[|X^{\rho}(t)-X^*(t)|^2+\int_t^T|Y^{\rho}(s)-Y^*(s)|^2ds +\int_t^T\|Z^{\rho}(s,\cdot)-Z^*(s,\cdot)\|_{\nu}^2ds\right] \\
&=2E\left[\int_t^T(X^{\rho}(s)-X^*(s))h(s)(dA^\rho(s)-dA^*(s))\right]\\
& \ \ \ +2E\left[\int_t^T(X^{\rho}(s)-X^*(s))(f^\rho(s)-f^*(s))ds\right] +E\left[\int_t^T (g^\rho(s)-g^*(s))ds\right]\\
& \ \ \ +E\left[\sum_{t<s\leq T}h(s)\int_{\mathcal{E}}(Z^\rho(s,e)-Z^*(s,e))\tilde{N}(\{s\},de)(\Delta A^\rho(s)-\Delta A^*(s))\right]\\
&\leq2\rho E\left[\int_t^T(X^{\rho}(s)-X(s))h(s)dA(s)\right]+\left(3C+\frac{1}{2}\right) E\left[\int_t^T(X^{\rho}(s)-X(s))ds\right]\\
& \ \ \ +\frac{1}{2}\left\{ \int_t^T|Y^{\rho}(s)-Y^*(s)|^2ds +\int_t^T\|Z^{\rho}(s,\cdot)-Z^*(s,\cdot)\|_{\nu}^2ds \right\}\\
& \ \ \ +\rho E\left[\sum_{t<s\leq T}h(s)\int_{\mathcal{E}}(Z^\rho(s,e)-Z^*(s,e))\tilde{N}(\{s\},de)\Delta A(s)\right].
\end{align*}
Rearranging the terms leads to
\begin{align*}
&E\left[|X^{\rho}(t)-X^*(t)|^2ds\right] +\frac{1}{2}\left\{ \int_t^T|Y^{\rho}(s)-Y^*(s)|^2ds +\int_t^T\|Z^{\rho}(s,\cdot)-Z^*(s,\cdot)\|_{\nu}^2ds \right\} \\
&\leq2\rho E\left[\int_t^T(X^{\rho}(s)-X(s))h(s)dA(s)\right]+\left(3C+\frac{1}{2}\right) E\left[\int_t^T(X^{\rho}(s)-X(s))ds\right]\\
& \ \ \ +\rho E\left[\sum_{t<s\leq T}h(s)\int_{\mathcal{E}}(Z^\rho(s,e)-Z^*(s,e))\tilde{N}(\{s\},de)\Delta A(s)\right].
\end{align*}
Letting $\rho\to 0$ and using Gronwall's lemma, we obtain
(\ref{7.2})-(\ref{7.3}). The result (\ref{7.1}), with the $\sup_{t\in[0,T]}$ inside the expectation, follows from the Burkholder-Davis-Gundy inequality.
\end{proof}

We now introduce the  variational equation：
\begin{align}\label{varequ}
\left\{
\begin{array}{lcl}
-dX^1(t)&=&\left[f_x^*(t)X^1(t)+f_y^*(t)Y^1(t)+\int_{\mathcal{E}}f_z^*(t, e)Z^1(t,e)\nu(de)\right]dt-Y^1(t)dW(t)\\
& & -\int_{\mathcal{E}}Z^1(t, e)\tilde
N(dt,de)+g_x^*(t)X^1(t)\overleftarrow{d}B(t)+h(t)dA(t),\\
X^1(T)&=&0,
\end{array}
\right.
\end{align}
where $f^*(t, e)=f(t, X^*(t), Y^*(t), Z^*(t,e))$ and $(X^1(t), Y^1(t), Z^1(t,e))$ are the derivative processes defined by
\begin{align*}
X^1(t)=\lim_{\rho\to0}\frac{1}{\rho}(X^\rho(t)-X^*(t)), \ \ \ t\in[0,T],
\end{align*}
\begin{align*}
 Y^1(t)=\lim_{\rho\to0}\frac{1}{\rho}(Y^\rho(t)-Y^*(t)), \ \ \ t\in[0,T],
\end{align*}
\begin{align*} Z^1(t,e)=\lim_{\rho\to0}\frac{1}{\rho}(Z^\rho(t,e)-Z^*(t,e)), \ \ \ t\in[0,T].
\end{align*}
Here, we  suppose that $\frac{\partial f}{\partial z}(t,e)>-1$, a.s. for all $(t,x,y,z)$.
 \begin{remark}
From Theorem \ref{TLGBDSDE}, we get
\begin{align}\label{SVAR}
X^1(t)=&E\left[\int_t^T\Gamma(t,s)h(s)dA(s) +\sum_{t<s\leq T}\Gamma(t,s)\int_{\mathcal{E}}f_z^*(s,e)N(\{s\},de)h(s)\Delta A(s)\Big|\mathcal{G}_t\right],  \ \ \ t\in[0, T],
\end{align}
where
\begin{align*}
\Gamma(t,s)=&\exp\left(\int_t^s\left\{f_x^*(r)-\frac{1}{2}g_x^*(r)^2-\frac{1}{2}f_y^*(r)^2\right\}dr+\int_t^sf_y^*(r)dW(r)
+\int_t^sg_x^*(r)\overleftarrow{d}B(r)\right.\notag\\
&\left. +\int_t^s\int_{\mathcal{E}}\ln (1+f_z^*(r,e))\tilde{N}(dr,de)+\int_t^s\int_{\mathcal{E}}\{\ln(1+f_z^*(r,e))-f_z^*(r,e)\}\nu(de)dr\right).
\end{align*}
\end{remark}

By the similarly calculation used in Wu and Liu \cite{Wu2018}, we can present the estimates of the perturbed state process $(X^{\rho},Y^{\rho},Z^{\rho})$.
\begin{lemma}\label{lem3}
\begin{align*}
&\lim_{\rho\to 0}
E\left[\sup_{t\in[0,T]}\Big|\frac{X^{\rho}(t)-X^*(t)}{\rho}-X^1(t)\Big|^2\right]=0,\\
&\lim_{\rho\to 0} E\left[\int_t^T\Big|\frac{Y^{\rho}(s)-Y^*(s)}{\rho}-Y^1(s)\Big|^2ds\right]=0, \ \ \ t\in[0,T],\\
&\lim_{\rho\to 0}
E\left[\int_t^T\Big\|\frac{Z^{\rho}(s,\cdot)-Z^*(s,\cdot)}{\rho}-Z^1(s,\cdot)\Big\|_{\nu}^2ds\right]=0,
\ \ \ t\in[0,T].
\end{align*}
\end{lemma}

\begin{lemma}\label{varine}
The following variational inequality
holds
\begin{align}\label{7.10}
E\left[\int_0^TF_x^*(t)X^1(t)+G_x^*(0)X^1(0)+\int_0^TH(t)dA(t)\right]\geq0.
\end{align}
\end{lemma}
\begin{proof}
From Lemma \ref{lem3} and using Taylor's expansion, we get
\begin{align*}
E\left[\int_0^TF_x^*(t)X^1(t)+G_x^*(0)X^1(0)+\int_0^TH(t)dA(t)\right]&=\frac{d}{d\rho}J(A^\rho(\cdot))\Big|_{\rho=0}\\
&=\lim_{\rho\downarrow0}\rho^{-1}[J(A^\rho(\cdot))-J(A^*(\cdot))]\\
&\geq0.
\end{align*}
\end{proof}

Now we can state the necessary maximum principle for the optimal singular control of
backward doubly stochastic system with general jumps.
\begin{theorem}\label{NSMP1}
We assert for almost all $t\in[0,T]$
\begin{align}\label{SME0}
H(t)-p^*(t)h(t)\left[1+\int_{\mathcal{E}}f_z^*(t,e)\tilde{N}(\{t\},de)\right]\geq 0, \ \  a.s.,
\end{align}
\begin{align}\label{SME00}
\left\{H(t)-p^*(t)h(t)\left[1+\int_{\mathcal{E}}f_z^*(t,e)\tilde{N}(\{t\},de)\right]\right\}\Delta A^*(t)= 0,\ \  a.s.,
\end{align}
\begin{align}\label{SME000}
H(t)-p^*(t)h(t)\geq0,\ \  a.s.,
\end{align}
\begin{align}\label{SME0000}
 \left[H(t)-p^*(t)h(t)\right]d{A^*}^c(t)=0, \ \  a.s.
\end{align}
\end{theorem}
\begin{proof}
Applying It\^{o}'s formula to $p^*(t)X^1(t)$, we find that
\begin{align*}
E[g_x^*(0)X^1(0)]=&-E[p^*(0)X^1(0)]\notag\\
=&-E\left[\int_0^Tp^*(t)h(t)dA(t)\right]-E\left[\int_0^TX^1(t)f_x^*(t)dt\right]\notag\\
&-E\left[\sum_{0<t\leq T}\int_{\mathcal{E}}f_z^*(t,e)p^*(t)h(t)\tilde{N}(\{t\},de)\Delta A(t)\right].
\end{align*}
Hence, by Lemma \ref{varine}, we obtain for all $A\in\mathcal A$
\begin{align}\label{SME}
E\left[\int_0^T(H(t)-p^*(t)h(t))dA(t)\right]-E\left[\sum_{0<t\leq T}\int_{\mathcal{E}}f_z^*(t,e)p^*(t)h(t)\tilde{N}(\{t\},de)\Delta A(t)\right]\geq 0.
\end{align}
In particular, fix $t\in[0,T]$ and choose $dA(s)=\alpha(\omega)\delta_t(s)$, where $s\in[0,T]$, $\alpha(\omega)\geq0$ is $\mathcal{F}_t$-measurable and bounded. Then (\ref{SME}) has the form
\begin{align*}
E\left[\left\{H(t)-p^*(t)h(t)\left[1+\int_{\mathcal{E}}f_z^*(t,e)\tilde{N}(\{t\},de)\right]\right\}\alpha\right]\geq 0.
\end{align*}
Since this holds for all bounded $\mathcal{F}_t$-measurable $\alpha$, we get (\ref{SME0}).

Furthermore,  choose $A(t)={A^*}^d(t)$, the purely discontinuous part of $A^*$. Then it follows from (\ref{SME}) that
\begin{align}\label{SME1}
E\left[\sum_{0<t\leq T}\left\{H(t)-p^*(t)h(t)\left[1+\int_{\mathcal{E}}f_z^*(t,e)\tilde{N}(\{t\},de)\right]\right\}\Delta A^*(t)\right]\geq 0.
\end{align}
On the other hand, choosing $A(t)=-{A^*}^d(t)$ in (\ref{SME}) gives
\begin{align}\label{SME2}
E\left[\sum_{0<t\leq T}\left\{H(t)-p^*(t)h(t)\left[1+\int_{\mathcal{E}}f_z^*(t,e)\tilde{N}(\{t\},de)\right]\right\}\Delta A^*(t)\right]\leq 0.
\end{align}
Combining (\ref{SME1}) and (\ref{SME2}) implies
\begin{align*}
E\left[\sum_{0<t\leq T}\left\{H(t)-p^*(t)h(t)\left[1+\int_{\mathcal{E}}f_z^*(t,e)\tilde{N}(\{t\},de)\right]\right\}\Delta A^*(t)\right]= 0.
\end{align*}
Since $A^*(t)$ is $\mathcal{F}_t$-measurable, from (\ref{SME0}) and $\Delta A^*(t)\geq0$, we obtain (\ref{SME00}).

In addition, choosing $dA(t)=a(t)dt$, $t\in[0,T]$, where $a(t)\geq0$ is continuous and  $a(t)$ is $\mathcal{F}_t$-measurable for a.e. $t\in[0,T]$, we arrive at
\begin{align*}
E\left[\int_0^T[H(t)-p^*(t)h(t)]a(t)dt\right]\geq0.
\end{align*}
Since this holds for all $\mathcal{F}_t$-measured processes $a(t)$, we get (\ref{SME000}).

Moreover, choosing $A(t)={A^*}^c(t)$, the continuous part of $A^*$, we derive from (\ref{SME}) that
\begin{align}\label{SME3}
E\left[\int_0^T(H(t)-p^*(t)h(t))d{A^*}^c(t)\right]\geq 0.
\end{align}
On the other hand, choosing $A(t)=-{A^*}^c(t)$ in (\ref{SME}) we deduce
\begin{align}\label{SME4}
E\left[\int_0^T(H(t)-p^*(t)h(t))d{A^*}^c(t)\right]\leq 0.
\end{align}
From (\ref{SME3}) and (\ref{SME4}), we get (\ref{SME0000}).
The proof is complete.
\end{proof}

Set
\begin{align*}
\mathcal{U}(t)=&h(t)\left[\int_0^tF_x^*(s)\Gamma(s,t)ds+\Gamma(0,t)G_x^*(0)\right]+H(t),\\
\mathcal{V}(t)=&h(t) \left[\int_0^tF_x^*(s)\Gamma(s,t)ds+\Gamma(0,t)G_x^*(0)\right]\left[1+\int_{\mathcal{E}}f_z^*(t,e)N(\{t\},de)\right]+H(t).
\end{align*}
Then we have the following results.
\begin{theorem}\label{NSMP2}
Suppose that $A^*$ is optimal for the problem (\ref{problem}). Then for almost all $t\in[0,T]$, we have
\begin{align*}
&\mathcal{U}(t)\geq0, \ \ \  \mbox{and} \ \ \  \mathcal{U}(t)d{A^*}^c(t)=0, \ \ \ a.s., \\
&\mathcal{V}(t)\geq0, \ \ \  \mbox{and} \ \ \  \mathcal{V}(t)\Delta A^*(t)=0,\ \ \ a.s.
\end{align*}
\end{theorem}
\begin{proof}
Note that
\begin{align}\label{VEQ}
\frac{d}{d\rho}J(A^\rho(\cdot))\Big|_{\rho=0}=E\left[\int_0^TF_x(t, X^*(t))X^1(t)\right] +E[G_x(X^*(0))X^1(0)]+E\left[\int_0^TH(t)dA(t)\right].
\end{align}
Substituting (\ref{SVAR}) to (\ref{VEQ}) gives
\begin{align*}
\frac{d}{d\rho}&J(A^\rho(\cdot))\Big|_{\rho=0}=E\left[\int_0^TH(t)dA(t)\right]+E\left[\int_0^TF_x^*(t)E\left[\int_t^T\Gamma(t,s)h(s)dA(s)|\mathcal{G}_t\right]dt\right]\\
&\ \ \ +E\left[\int_0^TF_x^*(t)E\left[\sum_{t<s\leq T}\Gamma(t,s)\int_{\mathcal{E}}f_z^*(s,e)N(\{s\},de)h(s)\Delta A(s)|\mathcal{G}_t\right]dt\right]\\
&\ \ \ +E\left[G_x^*(0)E\left[\int_0^T\Gamma(0,s)h(s)dA(s)+\sum_{0<s\leq T}\Gamma(0,s)\int_{\mathcal{E}}f_z^*(s,e)N(\{s\},de)h(s)\Delta A(s)|\mathcal{G}_t\right]\right]\\
&=E\left[\int_0^TH(t)dA(t)\right]+E\left[\int_0^Th(t)\left(\int_0^tF_x^*(s)\Gamma(s,t)ds\right)dA(t)\right]\\
&\ \ \ +E\left[\sum_{0<t\leq T}h(t)\int_0^tF_x^*(s)\Gamma(s,t)ds\int_{\mathcal{E}}f_z^*(t,e)N(\{t\},de)\Delta A(t)\right]\\
&\ \ \ +E\left[\int_0^TG_x^*(0)\Gamma(0,t)h(t)dA(t)+\sum_{0<t\leq T}\Gamma(0,t)\int_{\mathcal{E}}f_z^*(t,e)N(\{t\},de)h(t)\Delta A(t)\right]\\
&=E\left[\int_0^T\mathcal{U}(t)dA^c(t)+\sum_{0<t\leq T}\mathcal{V}(t)\Delta A(t)\right].
\end{align*}
Proceeding as in the proof of Theorem \ref{NSMP1}, we complete the proof.
\end{proof}
\begin{remark}
From Theorem \ref{NSMP1} and Theorem \ref{NSMP2}, we can get
\begin{align*}
p^*(t)=-\int_0^tF_x^*(s)\Gamma(s,t)ds-\Gamma(0,t)G_x^*(0).
\end{align*}
\end{remark}


\bibliographystyle{elsarticle-num}
\bibliography{SVS}

\end{document}